\documentclass[11pt,reqno]{amsart}

\usepackage{amssymb}
\usepackage{epsfig}
\usepackage{graphics,color}
\usepackage{hyperref}
\usepackage{tikz}
\usetikzlibrary{calc}
\usetikzlibrary{arrows}
\usepackage{bbm}
\usepackage{multirow}

\usepackage{rotating}

\newcommand\Z{\mathbb Z}

\newcommand{\ds}{\displaystyle}
\newcommand{\pp}{projection principle}
\newcommand\bp{B}
\newcommand{\vac} {\ensuremath{\circ}}
\newcommand{\occ} {\ensuremath{\bullet}}

\def\Prob{{\mathbb {P}}}
\DeclareMathOperator{\ssyt}{SSYT}
\DeclareMathOperator{\ssy}{SS}

\renewcommand\P{\operatorname{\mathbb P{}}}

\newcommand{\om} {\Omega}
\newcommand{\ofm} {\Omega^{FM}}

\newcommand{\be}{\begin{equation}}
\newcommand{\ee}{\end{equation}}

\newcommand{\refT}[1]{Theorem~\ref{#1}}

\newcommand{\refL}[1]{Lemma~\ref{#1}}
\newcommand{\refP}[1]{Proposition~\ref{#1}}

\newcommand{\refS}[1]{Section~\ref{#1}}
\newenvironment{proofof}[1]{\medskip\noindent
   \textit{Proof of #1.} }{\hfill $\qed$\par\medskip}

\newenvironment{romenumerate}[1][0pt]{
\addtolength{\leftmargini}{#1}\begin{enumerate}
 }{\end{enumerate}}

\newtheorem{theorem}{Theorem}[section]
\newtheorem{lemma}[theorem]{Lemma}
\newtheorem{proposition}[theorem]{Proposition}
\newtheorem{corollary}[theorem]{Corollary}

\newtheorem{remark}[theorem]{Remark}

\newtheorem{problem}[theorem]{Problem}

\newtheorem{conjecture}[theorem]{Conjecture}

\begin{document}

\title
{Correlations in the Multispecies TASEP and a Conjecture by Lam}

\author{Arvind Ayyer} 
\address{Arvind Ayyer\\
Department of Mathematics, Indian Institute of Science,\\ Bangalore - 560012, India}
\email{arvind@math.iisc.ernet.in}

\author{Svante Linusson} 
\address{Svante Linusson, Department of Mathematics, KTH-Royal Institute of Technology, 
  SE-100 44, Stockholm, Sweden.}
\email{linusson@math.kth.se}

\date{\today}  

\begin{abstract}
We study correlations in the multispecies TASEP on a ring.
Results on correlation of two adjacent points prove two conjectures by Thomas Lam on \\
(a) the limiting direction of a reduced random walk in $\tilde A_{n-1}$ and\\ 
(b) the asymptotic shape of a random integer partition with no hooks of length $n$, a so called $n$-core.

We further investigate two-point correlations far apart and three-point nearest neighbour correlations and prove explicit formulas in almost 
all cases. These results can be seen as a finite strengthening of correlations in the TASEP speed process by Amir, Angel and Valk\'o. 
We also give conjectures for certain higher order nearest neighbour correlations.
We find an unexplained independence property (provably for two points, conjecturally for more points) between points that are closer in position 
than in  value that deserves more study.
\end{abstract}

\maketitle
\section{Introduction} 
\label{S:Intro}
Probabilistic processes on combinatorial structures are of much current interest. 
Thomas Lam initiated the study of infinite random reduced words in affine Weyl groups in \cite{lam}. This can equivalently be formulated 
as a random walk on the affine Coxeter arrangement conditioned never to cross the same hyperplane twice.
He first states a remarkable formula for the limiting direction of the random walk on the Weyl alcoves in terms of a certain finite 
Markov chain on the underlying finite Weyl group, see \refT{T:Lam2} below.
He further specializes to the affine Weyl group $\tilde{A}_{n-1}$ and conjectures a closed formula for the limiting direction. 
Via known bijections (see, for example \cite{llms}) between affine Grassmannians and so-called $n$-cores, a special class of integer partitions, he also conjectures a limit shape  for the natural growth process of these partitions, see Theorem \ref{T:cores}. 
This limit shape is a natural finite version of the famous limit shape by Rost, \cite{Rost, Joh}.
The first purpose of this paper is to prove these two conjectures.

Lam states his conjectures in terms of a certain Markov chain on the set of permutations. This chain turns out to be equivalent, see 
\cite{AL}, to a Markov chain, which is a  multispecies 
variant of the Totally Asymmetric Simple Exclusion Process (or TASEP) on a ring. 
Unknown to Lam, that Markov chain was already studied by probabilists and statistical physicists 
and the stationary distribution was given an explicit interpretation in terms of queueing theory by Ferrari and Martin \cite{FM1,FM2}. 

We will prove Lam's conjectures by studying certain two-point correlations in the multispecies TASEP on a ring. 
In another direction, it turns out that Amir, Angel and Valk\'o
\cite{AAV} have studied the correlations in an infinite volume limit of the multispecies TASEP, called the {\em TASEP speed process}.
Motivated by the study of correlations in the three-species exclusion process on $\Z$ \cite{MG,FGM} they gave exact formulas for various marginals in the TASEP speed process. In particular, they studied two-point and three-point nearest-neighbour correlations in great detail as well as two-point correlations a distance apart. 

This naturally leads us to the second purpose of this paper, namely to further explore correlations in the multispecies TASEP.
Apart from the two-point correlations above, we find a closed formula, \refT{T:dinY} for the correlation of two points further apart. 
In important special cases, the formula simplifies to extremely simple expressions. We also give closed formulas for correlation of three adjacent particles in five of the six possible cases. The correlation formulas shows a remarkable independence property, provably for two particles and conjecturally for more particles, see Remark \ref{R:indep}, Corollary \ref{C:uniform} and the discussion in \refS{S:open}. Our proofs are of a somewhat intricate combinatorial nature relying on the
multiline queues of Ferrari-Martin. We would be very interested to see a more conceptual proof that could give some understanding and hopefully prove some of our conjectures. In a certain sense, several of the formulas by Amir, Angel and Valk\'o \cite{AAV} are limits of our formulas on correlations, which thus are finite strengthenings of theirs.

The structure of the paper is as follows. Section~\ref{S:back} contains background information: Section~\ref{S:TASEP} defines the multispecies TASEP and the multiline process, Section~\ref{S:redwords} defines infinite reduced words in $\tilde A_n$ and Section~\ref{S:ncore} defines $n$-cores. 
In Section \ref{S:LimitTheorems} we present the  conjectures by Lam about the limiting objects, first the limiting direction of the random walk in \refS{S:LimitDirection} and the limiting shape of $n$-cores in Section~\ref{S:LimitShape}.
We also prove there that Lam's conjectures follows from our computations about the nearest neighbour two-point correlations in Section \ref{S:Two adjacent}, where we also explain the connection to the TASEP speed process. 
We continue our study of two-point correlations far apart in Section~\ref{S:Two far}. Proofs of the formulas there will require enumerative formulas for certain constrained semistandard Young tableaux with two columns, which we relegate to Section~\ref{S:Y}. We finally study nearest neighbour three-point correlations in Section~\ref{S:threept}. We end with some conjectures and open problems in Section~\ref{S:open} deserving of future study.

\section*{Acknowledgements}
We thank the MSRI for hospitality, where this research was initiated.
We thank Thomas Lam, James Martin, Greta Panova, Anne Schilling and Lauren Williams for fruitful discussions.

\section{Background}
\label{S:back}

\subsection{Multispecies TASEP and Multiline queues}
\label{S:TASEP}

In this section we describe the multispecies TASEP  and the multiline queue process defined by Ferrari-Martin \cite{FM1} that we will use. We will borrow notation from \cite{AL}, but we will describe the homogeneous model that we are interested in rather than the more general inhomogeneous model. The interested reader should consult \cite{AL} for more details. 

The multispecies version of the totally asymmetric simple exclusion process (or TASEP for short) is a stochastic process which can be defined on an arbitrary directed graph, but we will only define it on directed path graphs of length $N$. The physical model is as follows. There are $N$ locations arranged in the shape of a ring, and each location is occupied by a particle. There are $n$ species of particles, labelled 1 through $n$ and $m_i$ particles of species $i$.
A multispecies TASEP model is thus determined by the $n$-tuple 
$m \equiv (m_1,\dots,m_n)$ such that $m_1+\cdots +m_n = N$. The space of configurations will be denoted $\Omega_m$.
It is easy to see that $|\Omega_m| = \ds \binom N{m_1,\cdots,m_n}$, and can be described naturally by multipermutations $S_m$, with $m_i$ repetitions of $i$, for $i \in [n]$. If each $m_i=1$, then the state space becomes that of ordinary permutations $S_n$.

Each particle carries an independent exponential clock which rings with rate 1. Whenever the clock for particle $j$ rings, 
it tries to exchange positions with the particle $k$ to its immediate left. The exchange is only successful if $k>j$ and fails otherwise. 
This Markov process is completely described by its generator, which is an $|\Omega_m| \times |\Omega_m|$ matrix $M_m$ labelled by multipermutations whose off-diagonal $(\sigma,\tau)$-entry describe the transition rate from $\tau \to \sigma$,
\[
M_m(\sigma,\tau) = 
1, \quad \text{if $\sigma_i = \tau_{i+1}> \tau_i = \sigma_{i+1}$ 
and $\sigma_j = \tau_j$ for $j \neq i,i+1$},
\]
for $\sigma \neq \tau$ and whose diagonal entries are negatives of the total
incoming transitions,
\[
M_m(\sigma,\sigma) = -\#\{\rho \in S_m, \rho \neq \sigma : M_m(\rho,\sigma) = 1\}.
\]
Let $\P_m(\sigma)$ denote the stationary probability of the multipermutation $\sigma$ of type $m$.
The following facts about the stationary probability distribution of the multispecies TASEP are well-known.
\begin{proposition}
\label{P:TASEPbasic}
\begin{romenumerate}
\item Let $\sigma' = (\sigma_2,\dots,\sigma_N,\sigma_1)$ denote the rotated version of $\sigma$. Then $\P_m(\sigma') = \P_m(\sigma)$.

\item Let $m^{\text{rev}} = (m_n,m_{n-1},\dots,m_1)$ and 
$\sigma^{\text{rev}} = (n+1-\sigma_N,\dots,n+1-\sigma_1)$. Then $\sigma^{\text{rev}} \in S_{m^{\text{rev}}}$ and 
$\P_{m^{\text{rev}}}(\sigma^{\text{rev}}) = \P_m(\sigma)$.

\item The stationary probability that the first site is occupied by species $i$ is given by $\P_m(w_1 = i) = m_i/N$.
\end{romenumerate}
\end{proposition}
The first property in Proposition~\ref{P:TASEPbasic} is called {\em rotational symmetry}, and follows simply because the Markovian dynamics is independent of the position. The second is a generalized version of what is sometimes called {\em particle-hole symmetry} although notice that the direction has also been switched. The last property follows from the first.

The following amazing result about the stationary probability distribution was proved by Ferrari and Martin.

\begin{theorem}[P. Ferrari and J. Martin, \cite{FM2}, Theorem 5.1] \label{T:tasepprob}
The stationary probability of multipermutations in $S_m$ are integer multiples of
\[
\left(\ds\prod_{i=1}^{n} \binom N{m_1+ \cdots + m_i}\right)^{-1}
\]
with the integer being 1 for the reverse permutation 
\[
\underbrace{n \cdots n}_{m_n} \underbrace{(n-1) \cdots (n-1)}_{m_{n-1}}  \dots
\underbrace{1 \cdots 1}_{m_1}.
\]
\end{theorem}

\refT{T:tasepprob} was proved by Ferrari and Martin by considering a larger Markov chain and projecting to the multispecies TASEP Markov chain by the procedure known as {\em lumping}. We now describe this larger model, which is called a {\em multiline queue}. Just as the multispecies TASEP is defined on a ring of circumference $N$, the multiline queue is defined on a discrete cylinder of circumference $N$ and height $n-1$, which can be thought of as a stack of $n-1$ such rings containing a total of $(n-1)N$ sites. Each site contains exactly one of two symbols; $\vac$ and $\occ$ called vacant and occupied.
Given an $n$-tuple $m \equiv (m_1,\dots,m_n)$ of positive integers summing to $N$, the set of multiline queues of type $m$, denoted $\ofm_m$,
includes all configurations of $\vac$'s and $\occ$'s such that the number of $\occ$'s at row $i$, enumerated from the top, 
is $M_i = \ds\sum_{j=1}^i m_j$.
Since the positions of $\occ$'s at each row are independent of one another, the total number of configurations is
\[
|\ofm_m| = \ds\prod_{i=1}^{n-1} \binom N{M_i},
\]
which is the same as the denominator in \refT{T:tasepprob} since $M_N = N$.

Ferrari-Martin \cite{FM2}, defined transitions between multiline queues that turned $\ofm_m$ into a Markov chain which was called 
the (inhomogeneous)
Ferrari-Martin multiline process \cite{AL}.

\begin{theorem}[P. Ferrari and J. Martin, \cite{FM2}, Theorem 3.1] 
\label{T:mlqprob}
The stationary distribution of the Ferrari-Martin multiline process is uniform on $\ofm_m$.
\end{theorem}

We do not need the precise definition of the transitions in the Ferrari-Martin multiline queues. The interested reader is referred to \cite{FM1}.  
See \cite{AL} for illustrative examples of both the multispecies TASEP and the multiline queueing Markov chains.
Several different set of transitions rules that also give uniform stationary distribution on $\ofm_m$ are given in \cite{LM}.

The connection between the multiline queues and the TASEP is expressed through the procedure formally known as lumping of Markov chains \cite[Lemma 2.5]{LPW}. 
We will now describe the procedure called {\em bully path projection} which relates the stationary distribution of the multispecies TASEP to the uniform
distribution on multiline queues. As expected from a projection procedure, this will be a surjective map $\bp: \ofm_m \to \Omega_m$.

Let $C \in \ofm_m$. The projection is defined recursively by bully paths.  A bully path
is a path going through the multiline queue which always moves rightwards or downwards and which contains exactly one $\occ$ from each row. Moreover, it moves downwards from a given row $j$ if and only if 
it has encountered a $\occ$ in row $j$ that has not already been part of another bully path.

We start by defining bully paths starting at locations $(1,i)$ (where we are using matrix notation for positions in multiline queues) such that $C_{1,i}=\occ$. The order in which we run these paths among these $\occ$'s turns out not to matter.
We mark by a 1 all the $m_{1}$ columns where the bully paths end at the bottom row. 
Next, we start bully paths at locations $(2,j)$ such that $C_{2,j} = \occ$ and
moreover which were not part of the bully paths from the first row. By the definition of multiline queues and the nature of the bully paths, there will be exactly $m_2$ of these.
We then mark all the $m_2$ columns in the bottom row that these paths reach by a 2 at row $n$. We continue in this way for all the $\occ$'s in all the rows. 
There will be exactly $m_j$ bully paths starting in row $j$ leading to a $j$ on row $n$ for $j \in [n-1]$. Finally, we mark the symbol $n$ on row $n$ below all
$k$'s such that the $C_{n-1,k} = \vac$.

At the end, the configuration on row $n$ of entries from $\{1,\dots,n\}$ is
precisely a multipermutation in $S_m$. This is the required configuration in $\om_{m}$. We call this projection $\bp$. We remark that $\bp$ is well-defined in the sense that the order of the bully paths starting at a given row do not matter.  See Figure \ref{F:bullypatheg} for an example.
The following theorem essentially states that the bully path projection is a lumping of Markov chains.

\begin{figure}[h!]
\begin{tikzpicture}[scale=1]
\matrix [column sep=0.5cm, row sep = 0.1cm, ampersand replacement=\&] 
{
\node {$\vac$}; \& \node{$\vac$}; \& \node(a1) {$\occ_{1}$}; \& \node {$\vac$}; \& \node {$\vac$}; \& \node {$\vac$}; \& 
\node(a21) {$\occ_1$}; \& \node {$\vac$};  \& \node {$\vac$}; \\
\node {$\vac$}; \& \node {$\vac$}; \& \node(a2) {$\vac$}; \& \node {$\vac$}; \& \node {$\vac$}; \& \node(a3) {$\occ_1$}; \& 
\node(a22) {$\occ_1$}; \& \node(b1) {$\occ_2$}; \& \node {$\vac$}; \\
\node(b3) {$\occ_2$}; \& \node(c1){$\occ_3$}; \& \node {$\vac$}; \& \node(c21) {$\occ_3$};\& \node {$\vac$}; \& \node {$\occ_1$};\& 
\node(a23) {$\vac$}; \& \node(b2) {$\vac$}; \& \node(a24) {$\occ_1$};\\ 
\node(a26) {$\occ_1$}; \& \node(b4){$\occ_2$}; \& \node {$\vac$}; \& \node(c3) {$\occ_3$}; \&\node(c23) {$\occ_3$}; \&\node {$\occ_1$};\& 
\node(d1) {$\occ_4$}; \& \node(d21) {$\occ_4$}; \& \node(a25) {$\vac$};\\
\node(a27) {$\occ_1$}; \& \node(b5){$\occ_2$}; \& \node {$\occ_{5}$}; \& \node(c4) {$\occ_3$}; \& \node(c24) {$\occ_3$};\&\node(a4) {$\occ_1$};\& 
\node(d2) {$\occ_4$}; \& \node(d22) {$\occ_4$}; \& \node {$\occ_{5}$};\\ 
};

\draw [-,thick,red] (a1.center) -- (a2.center) -- (a3.center) -- (a4.center);
\draw[-,thick,red] (a21.center) -- (a22.center) -- (a23.center) -- (a24.center) -- (a25.center)--(5,-0.6); 
\draw [-,thick, red] (-5,-0.6) -- (a26.center) -- (a27.center);
\draw [-,thick, blue] (b1.center) -- (3.35,0.05) -- (5,0.05);
\draw [-,thick, blue] (-5,0) -- (b3.center)  -- (a26.center) -- (b4.center) -- (b5.center);
\draw [-,thick,green] (c1.center) -- (b4.center) -- (c3.center) -- (c4.center);
\draw [-,thick,green] (c21.center) -- (c3.center) -- (c23.center) -- (c24.center);
\draw [-,thick] (d1.center) -- (d2.center);
\draw [-,thick] (d21.center) -- (d22.center);
\end{tikzpicture}
\caption{A multiline queue for $N=9, m=(2,1,2,2,2)$, with 
the bully-path projection to a multipermutation.}
\label{F:bullypatheg}
\end{figure}
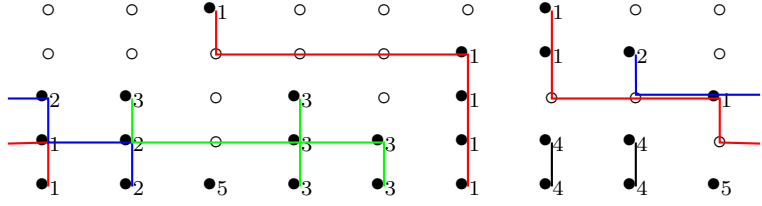

\begin{theorem}[P. Ferrari and J. Martin \cite{FM2}, Theorem 4.1]
\label{T:bullypath}
The $n$th line of the multiline queue process on $\ofm_m$ constructed by the bully path projection is exactly the multispecies TASEP on $\Omega_m$.
\end{theorem}

Therefore one way to understand correlations in the multispecies TASEP is by a combinatorial understanding of the multiline queues which give rise to those multipermutations which contribute to the correlations. In particular, the stationary probability of a multipermutation is given as follows.

\begin{corollary}
The stationary probability of a multipermutation $\pi$ in the TASEP is equal to

\[ \frac{\#\{q\in \ofm_m:B(q)=\pi\}}{\prod_{i=1}^{n-1} \binom N{M_i}}.
\]
\end{corollary}

\subsection{Infinite reduced words in $\tilde A_{n-1}$}
\label{S:redwords}

An infinite reduced word in the affine Weyl group $\tilde A_{n-1}$ is a word $\dots s_{i_5}s_{i_4}s_{i_3}s_{i_2}s_{i_1}$, where $s_{i_j}$'s, $0\le i_j\le n$ are generators and such that the word is infinite to the left and all finite initial sequences $s_{i_k}s_{i_{k-1}}\dots s_{i_2}s_{i_1}$ are reduced words of the affine Weyl group.
This can be seen as a walk on the alcoves of the arrangement corresponding to $\tilde A_{n-1}$, i.e. $\{x_i-x_j=d : 1\le i<j \le n, d\in \mathbb Z\}$ such that it never crosses the same hyperplane twice (starting in the fundamental alcove $x_{n}+1>x_1>x_2>\dots>x_{n}$). Such a walk is called reduced and we are interested in random reduced walks, i.e. at each step choosing uniformly one of the legal hyperplanes to cross. See Figure \ref{fig:walk} for an example of a walk that stays in the fundamental Weyl chamber.
From general Coxeter theory, it is not difficult to see that a reduced random walk is a Markov process, i.e. it only matters where the walk is stationed currently.

\begin{figure}[h!]
\begin{center}
\begin{tikzpicture}[scale=0.7]
	\clip (0,1) rectangle (8,9);
\begin{scope}[shift={($(1.15,0)$)}]
\draw[very thick] ($4*(-4, 1)$) -- ($(10,4*1)$);
\draw[very thick] ($(0,10)$) -- ($(0,0) + 10*(0.58,0)$);
\draw[very thick] ($(6.9,10)$) -- ($(6.9,0) - 10*(0.58,0)$);

\filldraw[fill=gray] ($(3*1.15,4)$) -- ($(3.5*1.15,5)$) -- ($(2.5*1.15,5)$);
\end{scope}

\begin{scope}[shift={($(4*1.15,4+2/3)$)},thick]
\draw[->] (0,0) -- ++ (90:2/3) -- ++ (150:2/3) -- ++ (90:2/3) -- ++ (150:2/3) -- ++ (90:2/3) -- ++(30:2/3)-- ++ (90:2/3) -- ++(150:2/3);
\end{scope}

\begin{scope}[shift={($(4*1.15,4)$)}]
\node [blue] at (0,1) {{ $0$}};
\node [blue] at ($(-1*1.15,1)$) {{ $1$}};
\node [blue] at ($(1*1.15,1)$) {{ $2$}};
\node [blue] at (0,-1) {{  $0$}};
\node [blue] at ($(-1*1.15,-1)$) {{ $1$}};
\node [blue] at ($(1*1.15,-1)$) {{ $2$}};

\node [blue] at ($(0.5*1.15,0)$) {{ $1$}};
\node [blue] at ($(-0.5*1.15,0)$) {{ $2$}};
\node [blue] at ($(1.5*1.15,0)$) {{ $0$}};
\node [blue] at ($(-1.5*1.15,0)$) {{ $0$}};

\foreach \he in {-2,-1,0,1}{
\node [blue] at ($(0.25*1.15,0.5+\he)$) {{ $2$}};
\node [blue] at ($(-0.25*1.15,0.5+\he)$) {{ $1$}};
\node [blue] at ($(0.25*1.15+0.5*1.15,0.5+\he)$) {{ $0$}};
\node [blue] at ($(0.25*1.15+1*1.15,0.5+\he)$) {{ $1$}};
\node [blue] at ($(0.25*1.15-1.5*1.15,0.5+\he)$) {{ $2$}};
\node [blue] at ($(0.25*1.15-1*1.15,0.5+\he)$) {{ $0$}};
}
\end{scope}

\begin{scope}[thin]
    \foreach \row in {0, 1, ...,10} {
      
        \draw ($\row*(0, 1)$) -- ($(10,\row*1)$);
 
    }
    \foreach \row in {-5,-4,-3,-2,-1,0,1,2,3,4,5,6,7,8,9,10} {
	\draw ($(\row*1.15,10)$) -- ($(\row*1.15,0) + 10*({1.15*0.5},0)$);
	}
   \foreach \row in {1,2,3,4,5,6,7,8,9,10,11,12,13,14,15,16} {
	\draw ($(\row*1.15,10)$) -- ($(\row*1.15,0) - 10*({1.15*0.5},0)$);
	}
\end{scope}
\begin{scope}[shift={($(4*1.15,4)$)}]
\draw[thick,red,->,dashed] (0:0) -- (90:4.7);
\node [red] at (0.4,4.4) {$\psi$};
\end{scope}
 \end{tikzpicture}
\caption{A reduced random walk on the alcoves of the $\tilde A_2$ arrangement.  The shown walk has reduced word 
$\cdots s_0s_2s_0s_1s_0s_2s_1s_0$. 
The thick lines divide into Weyl chambers.
A random walk staying in the fundamental chamber will almost surely be asymptotically parallel to the red dashed line.  }
\label{fig:walk}
\end{center}
\end{figure}
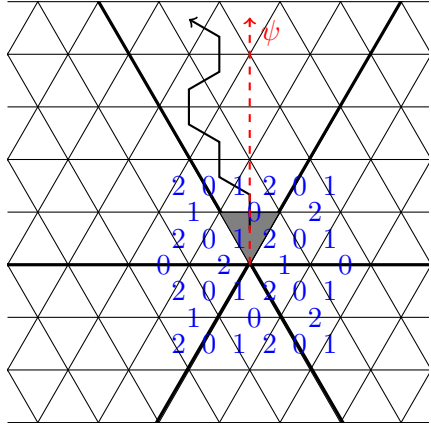

The study of reduced words of affine Weyl groups is part of a larger goal to try to lift results from the finite groups to the affine situation.
In \cite{lam} Thomas Lam proves that a reduced random walk will with probability one 'get trapped' in one of the chambers of the underlying 
finite Weyl group $A_{n-1}$ (which is the symmetric group) after a finite time. 
He also proves that in each  chamber $w\in A_{n-1}$ there exists a vector $\psi_w$ such that the walk will almost surely go in the direction of $\psi_w$. To be more precise:

\begin{theorem} [Lam \cite{lam}]\label{T:Lam}
Let $(X_0,X_1,\ldots)$ be a reduced random walk in $\tilde A_{n-1}$.  There exists a unit vector $\psi$ so that almost surely we have
\begin{equation}
\label{E:main}
\lim_{N \to \infty}v(X_i) \in A_{n-1} \cdot \psi
\end{equation}
where $v(X_i)$ denotes the unit vector pointing towards the central point of the alcove $X_i$.
\end{theorem}

Lam also proves, and this is the deepest theorem in the paper, that the Markov chain can be projected down to a certain Markov Chain $\om$ on $A_{n-1}$, with the following properties. If we let $\zeta(w)$ be the stationary distribution for $w\in A_{n-1}$ for $\om$, then $\zeta(w)$ is also the probability that the reduced random walk ends up in chamber $w^{-1}w_0$. The chain $\om$ is equivalent to the TASEP we discussed in this papper. Furthermore he proved that the distribution $\zeta$ also determines the direction of the vector $\psi$.

\begin{theorem}[Lam \cite{lam}]\label{T:Lam2}
The vector $\psi$ of Theorem \ref{T:Lam} is given by
$$
\psi = \frac{1}{Z}\sum_{w \in A_{n-1} \;:\; r_\theta w > w}  \zeta(w) w^{-1}(\theta^\vee).
$$
where $\theta$ is the highest root of $A_{n-1}$ and $Z$ is a normalization factor.  Furthermore, 
$$
\Prob(X \in C_w) = \zeta(w^{-1}w_0).
$$
\end{theorem}

He also stated the following conjecture which we prove he offered the following conjectures, where $\rho$ is the sum of all positive roots.

\begin{conjecture}[{Lam \cite[Conjecture 2]{lam}}]\label{conj:A}
For $\tilde A_{n-1}$, $\psi = \alpha\rho$ for some $\alpha > 0$.
\end{conjecture}

We will prove this conjecture in \refS{S:LimitTheorems}.

\subsection{$n$-cores}
\label{S:ncore}
The $n$-cores are special Young diagrams which come up in the study of the affine Grassmanian in algebraic combinatorics, see e.g. \cite{llms},
and can be defined in several equivalent ways. We will define these in terms of the hook length, but one can also define them in terms of ribbons.

Recall that a partition $\lambda = (\lambda_1, \dots, \lambda_r)$, whose entries are positive and weakly decreasing, can be represented as an array of left-justified boxes (called cells) such there are $\lambda_1$ cells at row 1, $\lambda_2$ cells at row 2, and so on. Such a representation is called a Young diagram, also denoted $\lambda$. We will take the convention that the rows are arranged from top to bottom in increasing order.
The {\em hook} of a cell $c$ in a Young diagram is the set of cells directly to its right and directly below it, as well as $c$ itself. The {\em hook length} of $c$ is the number of elements in its hook. An $n$-core then, is a partition that contains no cell whose hook length is divisible by $n$. Figure \ref{F:core} shows an example of a 4-core.

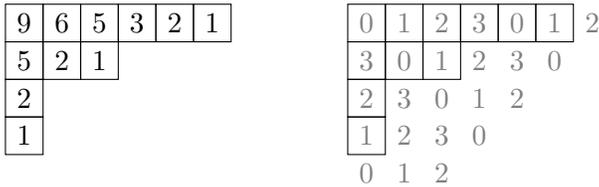
\begin{figure}[h]
\begin{tikzpicture}[scale=0.5]
\draw (0,0) rectangle (1,1);
\draw (0,1) rectangle (1,2);
\draw (1,2) rectangle (2,3);
\draw (0,2) rectangle (1,3);
\draw (0,3) rectangle (1,4);
\draw (1,3) rectangle (2,4);
\draw (2,2) rectangle (3,3);
\draw (2,3) rectangle (3,4);
\draw (3,3) rectangle (4,4);
\draw (4,3) rectangle (5,4);
\draw (5,3) rectangle (6,4);
\node at (0.5,0.5) {1};
\node at (2.5,2.5) {1};
\node at (5.5,3.5) {1};
\node at (4.5,3.5) {2};
\node at (0.5,1.5) {2};
\node at (1.5,2.5) {2};
\node at (3.5,3.5) {3};
\node at (2.5,3.5) {5};
\node at (1.5,3.5) {6};
\node at (0.5,2.5) {5};
\node at (0.5,3.5) {9};
\node at (0.5,-0.5) {\phantom{0} };
 \end{tikzpicture}\qquad\qquad
 \begin{tikzpicture}[scale=0.5]
\draw (0,0) rectangle (1,1);
\draw (0,1) rectangle (1,2);
\draw (1,2) rectangle (2,3);
\draw (0,2) rectangle (1,3);
\draw (0,3) rectangle (1,4);
\draw (1,3) rectangle (2,4);
\draw (2,2) rectangle (3,3);
\draw (2,3) rectangle (3,4);
\draw (3,3) rectangle (4,4);
\draw (4,3) rectangle (5,4);
\draw (5,3) rectangle (6,4);
\node [gray] at (0.5,3.5) {0};
\node [gray] at (1.5,3.5) {1};
\node [gray] at (2.5,3.5) {2};
\node [gray] at (3.5,3.5) {3};
\node [gray] at (4.5,3.5) {0};
\node [gray] at (5.5,3.5) {1};
\node [gray] at (6.5,3.5) {2};
\node [gray] at (0.5,2.5) {3};
\node [gray] at (1.5,2.5) {0};
\node [gray] at (2.5,2.5) {1};
\node [gray] at (3.5,2.5) {2};
\node [gray] at (4.5,2.5) {3};
\node [gray] at (5.5,2.5) {0};
\node [gray] at (0.5,1.5) {2};
\node [gray] at (1.5,1.5) {3};
\node [gray] at (2.5,1.5) {0};
\node [gray] at (3.5,1.5) {1};
\node [gray] at (4.5,1.5) {2};
\node [gray] at (0.5,0.5) {1};
\node [gray] at (1.5,0.5) {2};
\node [gray] at (2.5,0.5) {3};
\node [gray] at (3.5,0.5) {0};
\node [gray] at (0.5,-0.5) {0};
\node [gray] at (1.5,-0.5) {1};
\node [gray] at (2.5,-0.5) {2};
 \end{tikzpicture}
\caption{Two copies of the same 4-core. To the left is the hook length stated in each box. To the right is the content modulo 4.}
\label{F:core}
\end{figure}

The growth model of a random $n$-core is the following. Each position $(i,j)$ (matrix indexing) in the quarter plane is marked by its content 
$j-i \pmod n$. A growth corner of a Young diagram is a square just outside the diagram were a box could be added and it would still be a diagram. At each time step an integer $t$ in $[0,n-1]$ is choosen uniformly at random. For every position that is a growth corner and with content $t\pmod n$ we then add a box to the diagram.
If, for example, $n=4$ and the random sequence of integers $0,2,3,1,2,3,0,1$ we would get the diagram in Figure \ref{F:core}. Note that the first time the integer 2 appears in the sequence no box is added. There is a well known bijection between $n$-cores and affine Grassmanian elements 
of $\tilde A_{n-1}$. The sequence in the growth model correspond to left-multiplication by the simple generators.
For more information on this model of growing $n$-cores and the relation to algebraic combinatorics, see \cite{llms}. 

Note that this is the natural generalization of the growth model studied by Rost \cite{Rost} and Johansson \cite{Joh}. As we will see in \refS{S:LimitShape} the limit shape matches also the limit shape in that situation.

\section{Limit Theorems} 
\label{S:LimitTheorems}
In this section we will prove Conjecture \ref{conj:A}, that is we will determine the limiting direction of reduced random walks in the affine weyl 
groups $\tilde A_{n-1}$. This also implies the exact limiting shape of partitions with no hooks of length $n$, see \refS{S:LimitShape}.

\subsection{Reduced random walks in $\tilde A_{n-1}$}
\label{S:LimitDirection}
Let $e_i$ denote the unit vector in the $i$th coordinate direction. The highest root in \refT{T:Lam2} is $\theta=e_1-e_n$. 
In our setting \refT{T:Lam2} means summing over all permutations with $w_n>w_1$, where $w^{-1}(\theta^\vee)=e_{w_1}-e_{w_n}$ and 
$\zeta(w)$ is the stationary distribution for $w$ in the TASEP over permutations on a ring. 

The important quantity is thus $\P(w_1=i,w_n=j)=\sum_{w:w_1=i,w_n=j} \zeta(w)$. By rotational symmetry, namely Proposition~\ref{P:TASEPbasic}(i),
we can switch to the first and second position of $w$.  
For $j>i$, let $E_{j,i}(n)=\P(w_1=j,w_2=i)$ and
$E_{i,j}(n)=\P(w_1=i,w_2=j)$ in the TASEP $\Omega_n$.
We rewrite Lam's formula in \refT{T:Lam2} as
\begin{equation}\label{E:psi2}
\psi = \frac{1}{Z}\sum_{j=2}^n\sum_{i=1}^{j-1}  E_{j,i} (e_i-e_j).
\end{equation}
Using this, we will be able to prove Conjecture \ref{conj:A}.

\begin{theorem} \label{T:Direction}
The unit vector $\psi=\psi_{id}$ for the limiting direction of a reduced random walk constrained to the fundamental chamber
is
\[\psi=\frac{1}{\sqrt{2\binom{n+1}{3}}}\sum_{k=1}^n (n+1-2k)e_k.
\]
\end{theorem}
\begin{proof}
Given Theorem~\ref{T:cji} the coefficient of $e_k$ in \eqref{E:psi2} is 
\[
\frac{1}{n\binom{n}{2}}\left(\binom{n+1-k}{2}-\binom{k}{2}\right)=\frac{n+1-2k}{\binom{n}{2}}.
\]
\end{proof}
Note that this can also be stated (as Lam did) as $\psi$ is the sum of all positive roots, i.e. $\psi=\alpha\sum_{1\le i<j\le n} e_i-e_j$, for some constant $\alpha$.

An intuitive way to understand \eqref{E:psi2} is that when particle $i$
jumps over particle $j$ this corresponds to a step in the walk (constrained to the fundamental chamber) crossing a hyperplane of the type $x_i-x_j=d$, for some integer $d$.

\subsection{Limit shape of random $n$-core}
\label{S:LimitShape}
Thanks to the work of Lam we deduce as a direct consequence of Theorem \ref{T:Direction} the limit shape for a random $n$-core. 
Let $C_n$ be the piecewise linear curve with vertices 
$v_i=\gamma(\binom{i}{2},\binom{n-i+1}{2})$, for $i=1..n$, where $\gamma=\frac{2\sqrt 6}{n\sqrt{n^2-1}}$ is a scaling constant to make the 
area 1 between the curve and the axes, see Figure \ref{F:corelimit}.

\begin{figure}[h]
\begin{center}
\begin{tikzpicture}[scale=0.8]
\draw[very thick] (0,0.5)--(0.5,2)--(1.5,3)--(3,3.5); 
\draw[thin](0,0)--(0,3.5)--(3.5,3.5);
\end{tikzpicture}
\end{center}
\caption{The limiting piecewise-linear curve $C_4$ for random 4-cores.}
\label{F:corelimit}
\end{figure}
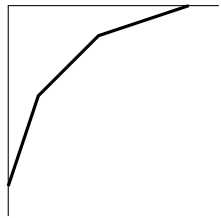

Let $D_n^K$ be the boundary of a random $n$-core after $K$ time steps scaled so the area of the $n$-core is 1.
Now, the theorem is that the curve $C_n$ is the limit shape of a random $n$-core.

\begin{theorem} 
\label{T:cores}
For each $\epsilon, \delta>0$ there is an $L$, such that for all $K>L$ we have
\[ \P(|D_n^K-C_n|>\delta)<\epsilon,
\]
where the distance between the curves is the supremum of the distances measured along the diagonals $y=-x+c$ for all $c$.
\end{theorem}
The proof of this theorem follows from Proposition 2 in \cite[Section 5]{lam} and \refT{T:Direction}. 

\begin{remark}
The classical limit shape for partitions, first proved by Rost \cite{Rost} in studies of first passage percolation,
is $\sqrt x +\sqrt {-y} =6^{1/4}$, scaled so the area is 1. Note that the vertices of $C_n$
converge to being on that shape, so the limit of $C_n$ as $n\to\infty$ is the classical shape. 
\end{remark}

\begin{remark} The slopes of the curve $D_n^K$ are deterministically $i/(n-i)$, i.e. it will have a pattern of $n-i$ steps to the left and then $i$ steps down for a long time until it changes to $n-i-1$ steps to the left and $i+1$ steps down. The position of this change correspond to the
vertex $v_{n-i}$ of the limit shape. The positions where the slope changes are random and converge to the $v_i$s.
\end{remark}

\section{Correlation of two adjacent particles} 
\label{S:Two adjacent}
For a particle $i$ in the TASEP $\om_n$ all the particles of a higher class look the same and all the particles of a lower class look the same. If we want to study the correlation of the particle of class $j$ and class $i, i<j$ in $\om_n$ we can study the correlation of  2 and 4 in the five species TASEP $\om_{m}$, where ${m}=(i-1,1,j-i-i,1,n-j)$. We call this the {\bf \pp{}} 
and it will be used repeatedly in this and the coming sections. Even though five species systems are easier to study than arbitrary $n$, they are often still too complicated to make precise calculations. We will go one step further and project to many three species systems. In this section we will use 
${m}_{s,t}=(s,t,n-s-t)$ to denote a system with $s$ 1's, $t$ 2's and remaining 3's. We will focus on all possible projections where $i>s$ and $j>s+t$ so that the particle of class $i$ (resp. $j$) will become a 2 (resp. 3). Studying the correlation between 2 and 3 in these systems will give us the correlation between particle $i$ and $j$. 

Let $T_{s,t}=\P(w_1=3,w_2=2)$ in $\om_{m_{s,t}}$. For $i<j$, recall that $E_{j,i}(n)=\P(w_1=j,w_2=i)$ and $E_{i,j}(n)=\P(w_1=i,w_2=j)$ in the TASEP $\om_n$. The following lemma is a consequence of the  \pp{}.

\begin{lemma} For all $0\le s,t<n$ with $s+t\le n$, 
\[
T_{s,t}=\sum_{j=t+s+1}^n\sum_{i=s+1}^{s+t} E_{j,i}(n).
\] 
\end{lemma}

By a standard inclusion-exclusion argument, we have 
\[
E_{j,i}(n)=T_{i-1,j-i}-T_{i,j-i-1}-\left (T_{i-1,j-i+1}+T_{i,j-i}\right ),
\]
and hence it is enough to compute $T_{s,t}$ to know $E_{j,i}$.
Computing the $E_{j,i}(n)$ for small values of $n$ reveals a clear pattern.
See Table \ref{Table:cji} for the values of $n\binom{n}{2}E_{j,i}$, when $n=5$.
The pattern in the lower left triangle of this table is easy to spot and for larger $n$ also the pattern in the upper right triangle.

\begin{table}[hbtdp]
\begin{center}
\begin{tabular}{|l||c|r|r|r|r|}
\hline
$w_1 \;\backslash w_2$ & 1& 2&3 & 4 & 5\\
\hline
\hline
1 & $0$ & $4$ & $2$ & $2$ & $2$ \\
\hline
2 & $1$ & $0$ & $5$ & $2$ & $2$ \\
\hline
3& $ 2$ & 1 & 0 & $5$ & 2\\
\hline
4& $ 3$ & 2& 1 & 0 & 4\\
\hline
5 &4& $ 3$ & 2& 1 & 0\\\hline
\end{tabular}
\vskip2mm
\caption{Values of $n\binom{n}{2}E_{w_1,w_2}$ for $n=5$.}\label{Table:cji}
\end{center}
\end{table}%

\begin{theorem}\label{T:cji}
For any $1\le i < j\le n$, we have 
\be \label{E:cji}
\begin{split} 
E_{j,i} &= \frac{j-i}{n\binom{n}{2}},\\
E_{i,j} & =\begin{cases}
\ds\frac{1}{n^2}+\frac{j(n-j)}{n^2(n-1)},& \quad\text{if $i=j-1$,}\\
\\
\ds\frac{1}{n^2},& \quad\text {if $i<j-1$.}\\
\end{cases}
\end{split}
\ee
\end{theorem}

\begin{remark}\label{R:indep}
Note that the probability that $j$ is followed by $i$ is the same for all $j<i-1$, is there an easy explanation for this? The probabilities $\P(w_1=j)=\P(w_2=i)=\frac 1n$, hence the probability $1/n^2$ could plausibly be interpreted as independence. 
\end{remark}

Before we go on to the proof, we recall some standard combinatorial
definitions. A semistandard Young Tableau (SSYT) \cite{EC2} of shape
$\lambda$ is a partition of shape $\lambda$ with positive integers in
the boxes that are strictly increasing along columns and weakly
increasing along rows.
Set $\ssyt_{r,k}(m)$ to be the number of semistandard Young
tableaux on a shape of two columns of lengths $r\ge l$ with no number
exceeding $m$.
An application of the hook-content
formula from standard combinatorial theory (see \cite[Corollary
  7.21.4]{EC2}) leads to
\be \label{hookcontent}
\ssyt_{r,k}(m)=
\begin{cases} 
\frac{r-k+1}{r+1}\binom{m}{r}\binom{m+1}{k}&, \text{ if $m\ge r\ge k\ge 0$} \\
0&, \text { otherwise}
\end{cases}
\ee

\begin{proofof}{Theorem~\ref{T:cji}}
 We will start with the case $E_{j,i}$. 
As described above we will use the projection principle and study the three species system $\om_{m_{s,t}}$, where ${m}_{s,t}=(s,t,n-s-t)$.
Since an easy calculation leads to
\begin{equation*}
\sum_{j=t+s+1}^n\sum_{i=s+1}^{s+t}  (j-i)
=t(n-s)(n-t-s)/2,
\end{equation*}
it suffices to prove that 
\be \label{E:21}
n\binom{n}{2}T_{s,t}=t(n-s)(n-t-s)/2.\ee

We will compute $T_{s,t}$ using the theory of multiline queues by Ferrari and Martin explained in Section~\ref{S:TASEP}. In the chain $\ofm_{m_{s,t}}$ we have $s$ particles (occupied sites) on the first row and $s+t$ particles on the second row. 
Note that if $s+t=n$, then there are no 3's so the formula \eqref{E:21} is trivially satisfied. Assume therefore that $s+t<n$ and $s,t\ge 1$.
Any multiline queue that projects to a word starting with a 3 and a 2 in second place must look necessarily have the following structure.
\be \label{E:que32}
\begin{array}{c c c c c c} \vac & \vac & .  & . & \dots & .\\ 
\vac & \occ & . & . & \dots & .\\ 
\hline 3 & 2 &  . & . & \dots & .\end{array}.
\ee
That is, no first class particle can be standing in queue to use the particle in the second position of the second row. This happens exactly if 

\[
\begin{array}{cccccccc}
  3 & \le z_{1,3} <  & \dots & < z_{1,t+2} <  & \cdots  & <  z_{1,n-s-1} <  & z_{1,n-s} & \le n\\
& &  & \text{\begin{sideways}  $\le$  \end{sideways}} &  \dots & 
\text{\begin{sideways}  $\le$   \end{sideways}} &  
\text{\begin{sideways}  $\le$   \end{sideways}} &  \\
& & 3 &  \le z_{2,2}  <  & \dots &  <  z_{2,n-s-t-1} < & z_{2,n-s-t}, &\\
\end{array}
\]
where $z_{a,b}$ is the position of the $b$'th vacant position in row $a$.  To
count these configurations, we use the connection with semistandard Young tableaux explained above.

The change of variables $\lambda_{1,b}=n+1-z_{1,n-s+1-b}$ and $\lambda_{2,b}=n+1-z_{2,n-s-t+1-b}$ shows that 
the number of multiline queues as in \eqref{E:que32} is equivalent to 
$\ssyt_{n-s-2,n-s-t-1}(n-2)$ given in \eqref{hookcontent}.

Since $|\ofm_{m_{s,t}}| = \binom{n}{s}\binom{n}{s+t}$, we get
\begin{align*}
\binom{n}{s}\binom{n}{s+t}T_{s,t}&=\frac{t}{n-s-1}\binom{n-2}{n-s-2}\binom{n-1}{n-s-t-1}\\
&=\binom{n}{s}\binom{n}{s+t}\frac{t(n-s)(n-s-1)(n-s-t)}{(n-s-1)n(n-1)n}.
\end{align*}

This simplifies to
\begin{equation*}
T_{s,t}=\frac{t(n-s)(n-s-t)}{n^2(n-1)}
\end{equation*}
 as wanted.
 There are two possibilities left. If $t=0$, there are no 2's and \eqref{E:21} is again trivially true. 
 If $s=0$ we have no 1's, $t$ 2's and $n-t$ 3's.
 The first row is thus only vacant positions and of all the $\binom{n}{t}$ ways of placing $t$ occupied positions on row 2, there are $\binom{n-2}{t-1}$
 which gives a word starting with a 3 and then a 2 as in \eqref{E:que32}. Thus 
  \[
T_{0,t}=\frac{\binom{n-2}{t-1}}{\binom{n}{t}}=\frac{t(n-t)}{n(n-1)},
 \]
again satisfying \eqref{E:21}.

\medskip
The cases $j<i$ are derived analogously, which we describe briefly. We use the \pp{ } to the same markov chain $\om_{m_{s,t}}$ as before.
In this case we can sum up the $E_{j,i}$ to get 
\[
\sum_{j=t+s+1}^n\sum_{i=s+1}^{s+t} E_{j,i}(n)=\frac{(n-s-t)(s+tn)}{n^2(n-1)}.
\]

We must thus prove that $\binom{n}{s}\binom{n}{s+t}T_{s,t}$, which is the number of multiline queues of the form
\be \label{E:que23}
\begin{array}{c c c c c c} \vac & . & .  & . & \dots & .\\ 
\occ & \vac & . & . & \dots & .\\ 
\hline 2 & 3 &  . & . & \dots & .\end{array},
\ee
is equal to 
$\binom{n}{s}\binom{n}{s+t}\frac{(n-s-t)(s+tn)}{n^2(n-1)}.$
This splits into two cases:
\begin{romenumerate}
\item There's no first class particle in queue above the 3,
\[
\begin{array}{c c c c c c} \vac & \vac & .  & . & \dots & .\\ 
\occ & \vac & . & . & \dots & .\\ 
\hline 2 & 3 &  . & . & \dots & .\end{array}
\]
This is identical to the previous case and is enumerated by 
\[
\ssyt_{n-s-2,n-s-t-1}(n-2)=\frac{t}{n-s-1}\binom{n-2}{n-s-2}\binom{n-1}{n-s-t-1}. 
\]

\item There is a first class particle in queue above the 3,
\[
\begin{array}{c c c c c c} \vac & \occ & .  & . & \dots & .\\ 
\occ & \vac & . & . & \dots & .\\ 
\hline 2 & 3 &  . & . & \dots & .\end{array}.
\]
Here we have one more vacant position to decide in the first row, which is the same as the first column of the SSYT being one longer. 
This case is therefore enumerated by 
\[
\ssyt_{n-s-1,n-s-t-1}(n-2)=\frac{t+1}{n-s}\binom{n-2}{n-s-1}\binom{n-1}{n-s-t-1}.
\] 
\end{romenumerate}
Adding the result of these two cases gives the desired result 
$\binom{n}{s}\binom{n}{s+t}T_{s,t}$ as wanted.
\end{proofof}

Using Theorem~\ref{T:cji}, we can obtain as a corollary, the joint distribution of the speeds of particles 0 and 1 in the TASEP speed Process 
\cite[Theorem 1.7]{AAV}. A brief explanation of how our results are related is as follows. The relevant model from \cite{AAV} is a TASEP on 
$\mathbb Z$ with classes uniformly taken in $[-1,1]$. 
The largest difference is that the TASEP studied here, $\om_m$ with $m=(1,\dots,1)$, is on a finite ring. Letting this ring grow to infinity the 
stationary distribution would converge to that of the line, see \cite{FM2}, and thus also the correlations. 
In their situation there are some obvious independences that one does not expect to find in the finite case but still mysteriously seem to be present; 
see discussion in Section \ref{S:open}.

\begin{corollary}\label{C:density2} [Amir-Angel-Valk\'o \cite{AAV}]

In the limit as $n \to \infty$, the probability mass function $E_{w_1,w_2}$ converges to the density function $f(x,y)+\mathbbm{1}_{\{x=y\}}g(x)$, where 
$g(x) = \frac{1-x^2}{8}$ and 
\[
f(x,y) = \begin{cases}
\frac{1}{4} & x>y,\\
\frac{y-x}{4} & x<y
\end{cases}
\]
\end{corollary}

\begin{proof}
We want to take the limit $w_1,w_2,n \to \infty$ at the same rate and rescale the resulting square so that it becomes $[-1,1]^2$. Following the convention in Theorem~\ref{T:cji}, we will let $i$ (resp. $j$) be the smaller (resp. larger) of $w_1$ and $w_2$. We thus have to take the limit so that
\[
\frac {i}n \to \frac{x+1}2 \text{ and }
\frac {j}n \to \frac{y+1}2 \text{ as $n \to \infty$}.
\]
A natural way of converting the probability mass function in \eqref{E:cji} into a probability density function in $[-1,1]^2$ is to divide the latter into $n^2$ smaller squares of area $4/n^2$ so that the $(w_1,w_2)$'th square contributes 
$E_{w_1,w_2}$ according to \eqref{E:cji}. This is naturally done by multiplying the values by $n^2/4$ and taking the limit $n \to \infty$.

When $w_1 = i < j-1 = w_2-1$, $n^2 E_{i,j}/4$ directly gives $1/4$. On the other hand,
\[
\frac{n^2}{4} E_{j,i} = \frac{j-i}{2(n-1)} \longrightarrow_{n \to \infty} \frac{y-x}{4},
\]
again as desired. However, this procedure does not work for $E_{j-1,j}$ because that line carries a nontrivial fraction of the mass. (Notice that
the scaling of $E_{j-1,j}$ in \eqref{E:cji} goes like $1/n$ for $j,n$ large.)
The resulting continuous measure on $[-1,1]^2$ is thus not absolutely continuous with respect to Lebesgue measure.

The correct way to take the limit is after multiplying by $\frac{n}{2}$ instead leading to
\[
\frac{n}{2} E_{j-1,j} = \frac{n-1+j(n-j)}{2n(n-1)}
\longrightarrow_{n \to \infty} \frac{1-y^2}{8},
\]
which becomes the singular continuous part of the density on the diagonal, completing the proof.
\end{proof}

Summing over all $j$ and $i$ we obtain exact corrections to the relative speeds between adjacent particles in the TASEP speed Process.

\begin{corollary}
\label{C:addprob12}
Let as before $w_1,w_2$ be the first two letters in the cyclic permutation in $\om_n$. Then probability at stationarity is exactly
\begin{align*} 
\P[w_1>w_2]  & = \frac{1}{3}+\frac{1}{3n}\\
\P[w_1=w_2-1]  & = \frac{1}{6}+\frac{7n-6}{6n^2}\\
\P[w_1<w_2-1]  & = \frac{1}{2}-\frac{3n-2}{2n^2},
\end{align*}
leading to the asymptotic result in  \cite[Theorem 1.7]{AAV}.
\end{corollary}

\section{Two-point correlations further apart}
\label{S:Two far}
In this section we will study the correlation between two positions of the TASEP that are a fixed distance apart.
Consistent with the previous section, let us denote
\[
E_{j,i}(1,a):=\P[w_1=j, w_a=i], 
\]
namely the probability that first letter is $j$ and the $a$'th letter is $i$ at stationarity in the multispecies TASEP. 
Although the notation is somewhat redundant, it avoids possible confusion between positions and labels.
Without loss of generality, we can take $j>i$ because if $j<i$, the theorem gives a formula by the rotational symmetry in Proposition~\ref{P:TASEPbasic}(i),
namely  $E_{j,i}(1,a) = E_{i,j}(1,n-a+2)$. 
Our main result here is an explicit formula for $E_{j,i}(1,a)$ in Theorem \ref{T:dinY}. 

An important ingredient in the proofs is the following object, whose study we
will undertake in Section~\ref{S:Y}.
\begin{align*}
Y_{r,l}^{\beta}(m)&:=\text{number of $\ssyt$s  on two columns $r \geq l$ with no entry exceeding} \\ 
& \text{$m$ such that the number $\beta$ appears somewhere in the second column}.
\end{align*}

\begin{theorem}\label{T:dinY} For $j>i$, 
\[
\begin{split} 
E_{j,i}(1,a) =& \frac{Y^{a-1}_{n-i,n-j+1}(n-1)}{\binom n{i-1} \binom n{j-1}}
- \frac{Y^{a-1}_{n-i,n-j}(n-1)}{\binom n{i-1} \binom n{j}} \\
&- \frac{Y^{a-1}_{n-i-1,n-j+1}(n-1)}{\binom n{i} \binom n{j-1}}
+ \frac{Y^{a-1}_{n-i-1,n-j}(n-1)}{\binom n{i} \binom n{j}}.
\end{split}
\]
\end{theorem}

We postpone the proof of Theorem~\ref{T:dinY} for later, but study some of its consequences for now. 
In the following situation, we get what is possibly the most striking result of this section.

\begin{corollary} \label{C:uniform}
For $1\le i<j\le n$ and $a\le j-i$ we have 
\[
E_{i,j}(1,a)=\frac{1}{n^2}.
\]
\end{corollary}
\begin{remark}
Note that $\P[w_1=j]=\P[w_a=i]=\frac{1}{n}$, so this could be interpreted as an independence when the difference in value is larger than the difference in position.
\end{remark}

We can follow the ideas of Corollary~\ref{C:density2} and derive the probability density when $i,j,n \to \infty$ for a fixed position $a$. This directly 
leads to the proof of \cite[Theorem 6.1(i)]{AAV}.

\begin{corollary}[Theorem 6.1(i), \cite{AAV}]
The joint probability density function $f(x,y)$ for the two-point correlation
$\mathbb{P}(w_1,w_a)$ when $a$ is fixed, $w_1<w_a$ and $w_1,w_a,n \to \infty$ is $1/4$ and consequently, $\mathbb{P}(w_1<w_a) = 1/2$ in the limit.
\end{corollary}

\begin{corollary} \label{C:dni}
For $2 \leq a \leq n$ and $1 \leq i \leq n-1$,
\[
E_{n,i}(1,a)=
\frac{1}{n^2}+\frac{(n-i)\binom{i-1}{a-2}-\binom{i-1}{a-1}}{a\, n\, \binom{n}{a}}.
\]
\end{corollary}

\begin{remark}
There are two simplification of Corollary~\ref{C:dni}.
\begin{romenumerate}
\item If $a>i+1$,
\(
\ds E_{n,i}(1,a)= \frac{1}{n^2}.
\)

\item If $i=n-1$,
\(
\ds E_{n,n-1}(1,a)= \frac{a-1}{n\binom{n}{2}}.
\)
\end{romenumerate}
\end{remark}

We will assume that $j>i$ and follow similar arguments as in Section~\ref{S:Two adjacent}. This time we will project down to the 
three-species system with sector 
${m_{x,y}}=(n-x-y,x,y)$ (note the slight difference with $m_{s,t}$ in Section~\ref{S:Two adjacent}). 
We will think of $j$ as one of the 3's and $i$ as one of the 2's.
Let $D_{y,x}(1,a)$ be the probability that there is a 3 in the first position and a 2 in position $a$ in this three 
species system. Then by the \pp{} we have 
\[
D_{y,x}(1,a):=\sum_{i=n-x-y+1}^{n-y}\sum_{j=n-y+1}^n E_{j,i}(1,a).
\] 
We can recover the $E_{j,i}(1,a)$'s from the $D_{y,x}(1,a)$'s by
 \be \label{e:recover}
E_{j,i}(1,a)=D_{n-j+1,j-i}(1,a)-D_{n-j,j-i+1}(1,a)
-D_{n-j+1,j-i-1}(1,a)+D_{n-j,j-i}(1,a).
\ee

We now claim the following.

\begin{lemma} \label{L:xylesseqn}
If $x+y \leq n$, then
\[
D_{y,x}(1,a)=\frac{Y_{y+x-1,y}^{a-1}(n-1)}{\binom n{y+x}\binom ny}.
\]
In the special case $x+y=n$, $D_{y,x}(1,a)$ becomes independent of $a$,
\[ 
D_{y,x}(1,a)=\frac{y(n-y)}{n(n-1)}
\]
\end{lemma}
\begin{proof}
The total number of multiline queues with sector ${m_{y,x}}$ is $\binom n{y+x}\binom ny$ since we can choose as vacant positions 
any $x+y$ positions in the first row and any $y$ positions in the second row. Hence $\binom n{y+x}\binom ny\cdot D_{y,x}(1,a)$ is the number of multiline queues on two rows that projects to a 3 in the first position and a 2 in position $a$. They must look like
\[
\begin{array}{c c c c c c} .& \dots & \vac &  . & \dots & .\\ 
\vac & \dots & \occ & . & \dots & .\\ 
\hline 3 & \dots & 2 &  . & \dots & .\end{array}, 
\]
where the 2 is in position $a$. Because of translation invariance of the multiline queue process, we can renumber positions so the 2 is in column $n$ instead. As in the proof of Theorem~\ref{T:cji} we let $z_{c,d}$ be the position of the $c$'th vacant position in row $d$. Then we have the following situation
\[
\begin{array}{cccccccc}
  1 & \le z_{1,1} <  & \dots & < z_{1,x} <  & \cdots  & <  z_{1,x+y-2}  <  & z_{1,x+y-1} & \le n-1\\
&  &  & 
 \text{\begin{sideways}  $\le$   \end{sideways}} &  \dots & \text{\begin{sideways}  $\le$   \end{sideways}} &   
 \text{\begin{sideways}  $\le$   \end{sideways}} &  \\
& & 1 &  \le z_{2,1} <  & \dots &  <  z_{2,y-1} < & z_{2,y}, &\\
\end{array}
\]
\noindent where $z_{2,d}=n-a+1$ for some $d$. Using the change of variables $\lambda_{1,b}=n-z_{1,x+y-b}$ and $\lambda_{2,b}=n-z_{2,y+1-b}$ gives a bijection to SSYT counted by $Y_{x+y-1,y}^{a-1}(n-1)$.

Although the first formula in Lemma~\ref{L:xylesseqn} specializes to the second one when $x+y=n$, the proof does not follow the same way. In fact, the proof is easier since we have projected to a two-species system $(x,n-x)$. The corresponding one line multiqueue has one specified occupied position and one specified vacant position. The other $x-1$ occupied positions can be chosen arbitrarily, so $\binom{n}{x}D_{y,x}(1,a)=\binom{n-2}{x-1}$.
\end{proof}

\begin{proofof}{Theorem~\ref{T:dinY}}
Using \refL{L:xylesseqn} and \eqref{e:recover}, we obtain the required formula for $E_{j,i}(k)$ whenever $j>i$. 
\end{proofof}

\begin{proofof} {Corollary~\ref{C:uniform}}
As noted above $E_{i,j}(1,a)=E_{j,i}(1,n+2-a)$, so we could rephrase the corollary as $E_{j,i}(1,a)=\frac{1}{n^2}$ when $a\ge i+2+n-j$ for $j>i$. 
It is this statement that we will prove. 

By particle-hole symmetry Proposition~\ref{P:TASEPbasic}(ii),
$E_{j,i}(1,a)=E_{n-i+1,n-j+1}(1,a)$. The equation in Theorem \ref{T:dinY}
can thus be rewritten as

\be 
E_{j,i}(1,a) 
=\frac{Y^{a-1}_{j-1,i}(n-1)}{\binom n{j} \binom n{i}}
- \frac{Y^{a-1}_{j-1,i-1}(n-1)}{\binom n{j} \binom n{i-1}}
- \frac{Y^{a-1}_{j-2,i}(n-1)}{\binom n{j-1} \binom n{i}}
+ \frac{Y^{a-1}_{j-2,i-1}(n-1)}{\binom n{j-1} \binom n{i-1}}.
\ee

Define 
\[
g_{j,i}(a):=\frac{Y^{a-1}_{j-1,i}(n-1)}{\binom n{i} \binom n{j}}
- \frac{Y^{a-1}_{j-2,i}(n-1)}{\binom n{i} \binom n{j-1}}.
\]  
Corollary~\ref{C:Yconst} below states that both $Y^{a-1}_{j-1,i}(n-1)$ and $Y^{a-1}_{j-2,i}(n-1)$ are independent of $a$ if $i+2+n-j \leq a$. This is precisely the condition that $a$ satisfies.
Therefore we can replace $a$ in both by $n-1$. Therefore, using \refP{P:Y1} below,
\begin{align*}
g_{j,i}(a)\! &= \frac{j \; Y^{n-1}_{j-1,i}(n-1) - (n+1-j) \; Y^{n-1}_{j-2,i}(n-1)}
{n \; \binom {n-1}{j-1} \binom n i}\\
&\! =  \frac{j \left[ \binom{n-1}{i-1} \binom{n-1}{j-1}- \binom{n-2}{i-2} \binom{n}{j} \right] - (n+1-j)\left[ \binom{n-1}{i-1} \binom{n-1}{j-2}- \binom{n-2}{i-2} \binom{n}{j-1} \right]}
{n \; \binom {n-1}{j-1} \binom n i}\\
&=  \frac{\binom{n-1}{i-1} \binom{n-1}{j-1}}
{n \; \binom {n-1}{j-1} \binom n i}\\
&= \frac{i}{n^2}.
\end{align*}
Then $E_{j,i}(1,a)=g_{j,i}(a)-g_{j,i-1}(a)$ implies the result.
\end{proofof}

\begin{proofof} {Corollary~\ref{C:dni}}
We use Theorem \ref{T:dinY} in the case $j=n$. Notice that the second and fourth terms are zero because they correspond to SSYT with only one column. Therefore, we have
\[
E_{n,i}(1,a) = \frac{Y^{a-1}_{n-i,1}(n-1)}{n\binom n{i-1}}
- \frac{Y^{a-1}_{n-i-1,1}(n-1)}{n\binom n{i}}.
\]
Let us look at the first term. The second term is similar.
We have to count SSYT with two columns whose column lengths are $n-i$ and 1, maximum entry $n-1$ and where $a-1$ sits in the second column. Among all $\binom{n-1}{n-i}$ choices for the entries in the first column, we can't accept the ones where the minimal entry is greater than $a-1$. Thus,
\[
Y^{a-1}_{n-i,1}(n-1) = \binom{n-1}{n-i} - \binom{n-a}{n-i},
\]
which leads to
\[
E_{n,i}(1,a) = \frac{\binom{n-1}{n-i} - \binom{n-a}{n-i}}{n\binom n{i-1}}
- \frac{\binom{n-1}{n-i-1} - \binom{n-a}{n-i-1}}{n\binom n{i}}.
\]
After a few manipulations, this leads to
\[
E_{n,i}(1,a) = \frac 1{n^{2}}+\frac 1{n a \binom na} \left[ (n-i) \binom i{a-1} - (n+1-i) \binom{i-1}{a-1} \right],
\]
which is equivalent to what we wanted to prove. 
\end{proofof}

\section{Constrained Semistandard Young Tableaux}
\label{S:Y}
In this section we compute the number of semistandard Young tableaux with two columns of lengths $r \geq l$, with maximum entry $m$ and such that a fixed $\beta$ lies in the second column. We denoted this $Y_{r,l}^{\beta}(m)$ in Section~\ref{S:Two far}. To that end, we introduce the following quantities. Let $X_{r}^{\alpha,\beta}$ be the number of SSYTs  of shape $r,r$, with entries $(\alpha,\beta)$ in the last row. Similarly, let $Z_{r,l}^{\alpha,\beta}(m)$ be the number of SSYTs  of shape $r,l$, with no entry exceeding $m$ such that the first row is 
 $(\alpha,\beta)$. We first give enumerative results for both these quantities.

\begin{lemma}\label{L:X}
For all $r\le\alpha\le\beta$ we have 
$X_{r}^{\alpha,\beta}=\binom{\beta}{r-1}\binom{\alpha-1}{r-1}-\binom{\beta-1}{r-2}\binom{\alpha}{r}$.
\end{lemma}
\begin{proof} Clearly true for $r=1$ since $X_{1}^{\alpha,\beta}=1$, as usual we define $\binom{n}{-1}=0$. The lemma then follows by induction over $r$ using the 
recursion $X_{r+1}^{\alpha,\beta}=\sum_{a=r}^{\alpha-1}\sum_{b=a}^{\beta-1} X_{r}^{a,b}$.
\end{proof}

\begin{lemma}\label{L:Z}
For all $\beta\le m$, $l\le r\le m$ we have
\[Z_{r,l}^{1,\beta}(m)=\sum_{1\le i\le j\le \beta}(-1)^{i+1}\binom{j-2}{i-2}\binom{\beta-j+i-1}{i-1}\ssyt_{r-i,l-i}(m-j).\]
\end{lemma}
\begin{proof} 
We will keep $r,l,m$ fixed through out the proof and introduce the shorter notation $\ssy_i(j):=\ssyt_{r-i,l-i}(m-j)$.
First we have the obvious identity 
$Z_{r,l}^{\alpha,\beta}(m)=Z_{r,l}^{1,\beta-\alpha+1}(m-\alpha+1)$. 
Second we will use the recursion 
\be \label{E:Zrec}
Z_{r,l}^{1,\beta}(m)=\ssy_1(1)-\sum_{a=2}^{\beta}\sum_{b=a}^{\beta} Z_{r-1,l-1}^{a,b}(m) ,
\ee
which comes from studying which pairs of numbers $a\le b$ that cannot be the entries in the second row and subtract off those cases.

Now, let us use the ansatz 
\be \label{E:Zansatz}
Z_{r,l}^{1,\beta}(m)=\sum_{1\le i\le j\le \beta} g_{i,j}^{\beta}\cdot \ssy_i(j),
\ee
for some coefficients $g_{i,j}^{\beta}$. 
Plugging the ansatz \eqref{E:Zansatz} into the recursion \eqref{E:Zrec} translates, after a few steps, to 
\[
\sum_{1\le i\le j\le \beta} g_{i,j}^{\beta}\cdot \ssy_i(j)=\ssy_1(1)-
\sum_{2\le a\le b\le \beta} \sum_{1\le i\le j\le b-a+1} g_{i,j}^{b-a+1}\cdot \ssy_{i+1}(j+a-1).
\]
This forces $g_{1,1}^{\beta}=1$ and gives the recursion
\[
g_{s,t}^{\beta}=-\sum_{r=s-1}^{t-1}\sum_{u=r}^{r+\beta-t} g_{s-1,r}^u.
\] 
Induction over $s$ now easily proves 
\[
g_{s,t}^{\beta}=(-1)^{s+1}\binom{t-2}{s-2}\binom{\beta-t+s-1}{s-1}, 
\]
which is the desired coefficient.
\end{proof}

We now consider special values of $\beta$, where the formula for 
$Y_{r,l}^{\beta}(m)$ is simple.

\begin{proposition}\label{P:Y1} 
\begin{romenumerate}
\item 
$\ds
Y_{r,l}^{1}(m) = \ssyt_{r-1,l-1}(m-1).
$
\item 
\[
Y_{r,l}^{m}(m) = \binom{m}{l-1}\binom{m}{r}
-\binom{m-1}{l-2}\binom{m+1}{r+1}.
\]
\end{romenumerate}
\end{proposition}

\begin{proof}
\begin{romenumerate}
\item The first row has to be $(1,1)$, and hence the minimum entry is two in the remainder of the SSYT.

\item $\beta=m$ can only be present in the second column in the last row. Let us sum over all possible entries in the first column at the same row. Using the definition of 
$X_{r}^{\alpha,\beta}$ and $Z_{r,l}^{\alpha,\beta}(m)$, we get
\[
Y_{r,l}^{m}(m) = \sum_{y=1}^{m} X_{l}^{y,m} Z_{r-l+1,1}^{y,m}(m).
\]
Since $Z_{r,1}^{\alpha,m}(m)$ is just the binomial coefficient $\binom{m-y}{r-l}$, we get
\begin{align*}
Y_{r,l}^{m}(m) &= \sum_{y=1}^{m} \binom{m-y}{r-l} \left( 
\binom m{l-1} \binom{y-1}{l-1} - \binom{m-1}{l-2}\binom{y}{l} \right), \\
&= \binom m{l-1} \sum_{y=1}^{m} \binom{m-y}{r-l} \binom{y-1}{l-1}- 
\binom{m-1}{l-2}\sum_{y=1}^{m} \binom{m-y}{r-l} \binom{y}{l}, \\
&= \binom m{l-1} \binom{m}{r} - 
\binom{m-1}{l-2} \binom{m+1}{r+1}.
\end{align*}
\end{romenumerate}
\end{proof}

\begin{theorem}
\label{T:Y2} 
For $m\ge r\ge l\ge 1$ and $1\le\beta < m$ we have 
\[
Y_{r,l}^{\beta}(m)=\sum_{1\le f\le e\le \beta} N_{e-1,f-1} \cdot \ssyt_{r-f,l-f}(m-e),
\]
where
\[
N_{e,f} = \frac{1}{e}\binom{e}{f+1}\binom{e}{f}
\]
 are the Narayana numbers.
\end{theorem}

\begin{proof} 
Summing over in which row $x$ of the last column that contains $\beta$ and which number $y$ is in row  $x$ of the first column,
 we get the following recursion.
\begin{align*}
Y_{r,l}^{\beta}(m) &=\sum_{x=1}^l\sum_{y=1}^{\beta} X_{x}^{y,\beta}Z_{r-x+1,l-x+1}^{y,\beta}(m), \\
&=\sum_{x=1}^l\sum_{y=1}^{\beta} X_{x}^{y,\beta}Z_{r-x+1,l-x+1}^{1,\beta-y+1}(m-y+1),
\end{align*}
where we have used \refL{L:Z} in the second line.
This leads to
\begin{align*}
&Y_{r,l}^{\beta}(m) = \sum_{x=1}^l\sum_{y=1}^{\beta} \left[\binom{\beta}{x-1}\binom{y-1}{x-1} -\binom{\beta-1}{x-2}\binom{y}{x}\right]  \\
 &\times\sum_{1\le i\le j\le \beta-y+1}(-1)^{i+1}\binom{j-2}{i-2}\binom{\beta-y-j+i}{i-1} 
  \ssy_{x+i-1}(y+j-1),
\end{align*}
where we have reused the notation $\ssy_i(j)$ from Lemma~\ref{L:Z} and fixed $r,l,m$

Note that all except the $y$-sum are natural in the sense that we can replace the limits
by $\Z$. The terms outside of the specified limits will then be zero by definition of the binomial coefficient.

As a first step, replace $i$ by $f=x+i-1$ and $j$ by $e=y+j-1$ to get
\begin{align*}
&Y_{r,l}^{\beta}(m) = \sum_{x}\sum_{y=1}^{\beta} \left[\binom{\beta}{x-1}\binom{y-1}{x-1} -\binom{\beta-1}{x-2}\binom{y}{x}\right]  \\
 &\times\sum_f \sum_{e}(-1)^{f+x}\binom{e-y-1}{f-x-1}\binom{\beta-e+f-x}{f-x} 
  \ssy_{f}(e).
\end{align*}
We are now going to move the $x,y$ sums inside the $e,f$ sums
\begin{align*}
&Y_{r,l}^{\beta}(m) = \sum_f \sum_{e}\sum_{x}\sum_{y=1}^{\beta} (-1)^{f+x}
\left[\binom{\beta}{x-1}\binom{y-1}{x-1} -\binom{\beta-1}{x-2}\binom{y}{x}\right]  \\
 &\times \binom{e-y-1}{f-x-1}\binom{\beta-e+f-x}{f-x} 
  \ssy_{f}(e).
\end{align*}
Now the last binomial coefficient is zero unless $\beta \geq e+x-f$, and the one before that is zero unless $e-y \geq f-x$. Therefore, we can replace the upper limit of the $y$-sum by 
$e+x-f$. Now we are in a position to do the $y$-sums. We use the identity
\[
\sum_{m=0}^{n} \binom mj \binom{n-m}{k-j} = \binom{n+1}{k+1}, \text{ if $0 \leq j \leq k 
\leq n$,}
\]
to obtain
\begin{align*}
\sum_{y=1}^{e+x-f} \binom{e-y-1}{f-x-1} \binom{y-1}{x-1} &= \binom{e-1}{f-1}, \\
\sum_{y=1}^{e+x-f} \binom{e-y-1}{f-x-1} \binom{y}{x} &= \binom{e}{f}.
\end{align*}
Note that it does not matter if the sum runs to $e+x-f$ or $e$. Thereafter we are left with
\begin{align*}
&Y_{r,l}^{\beta}(m) = \sum_f \sum_{e}   \ssy_{f}(e)  \\
 &\times\sum_{x} (-1)^{f+x}  \binom{\beta-e+f-x}{f-x}
\left[\binom{\beta}{x-1}\binom{e-1}{f-1} -\binom{\beta-1}{x-2}\binom{e}{f}\right].
\end{align*}
Now the $x$-sum can be done by first using $\binom{n}{k}=(-1)^{k}\binom{k-n-1}{k}$ and then variants of the Chu-Vandermonde identity,
\begin{align*}
\sum_{x=1}^{f}(-1)^{x} \binom{\beta-e+f-x}{f-x}\binom{\beta}{x-1} 
&= (-1)^{f}\binom{e-1}{f-1}, \\
\sum_{x=2}^{f} (-1)^{x+1}\binom{\beta-e+f-x}{f-x} \binom{\beta-1}{x-2} 
&= (-1)^{f+1}\binom{e-2}{f-2}.
\end{align*}
Noting that
\[
\binom{e-1}{f-1}^{2}-\binom ef \binom{e-2}{f-2} = N_{e-1,f-1},
\]
the Narayana number, completes the proof.
\end{proof}

\begin{corollary} \label{C:Yconst} 
For $l \leq r$ and $l+m-r \leq \beta \leq m$, $Y_{r,l}^{\beta}(m)$ is independent of 
$\beta$.
\end{corollary}

\begin{proof}
We use the formula for $Y^{\beta}_{r,l}(m)$ in \refT{T:Y2}.
For clarity, we rewrite the formula and explicitly include the binomial coefficients.
\[
Y_{r,l}^{\beta}(m)\! =\! \! \!\sum_{1\le f\le e\le \beta}  \frac{r-l+1}{(e-1)(r-f+1)}
\binom{e-1}{f}\binom{e-1}{f-1} \binom{m-e}{r-f}\binom{m-e+1}{l-f},
\]
First of all, $f \leq l$ by the last binomial coefficient.
More importantly, $e \leq m-r+f$, by the penultimate one. Using the constraint on $\beta$, 
we get that
\[
m+f-r \leq  m+l-r \leq \beta.
\]
Therefore we can replace the upper limit of $e$ by $m+l-r$ in the summation.
\end{proof}

\begin{remark} There is a nice combinatorial way of proving Corollary \ref{C:Yconst} directly. Given $l+m-r < \beta \leq m$ we can define a bijection
from $Y_{r,l}^{\beta}(m)$ to $Y_{r,l}^{\beta-1}(m)$ as follows. Assume $\lambda\in Y_{r,l}^{\beta}(m)$ with $\lambda_{2,b}=\beta$ for some row 
$b\le l$.
There are two cases. 
If  $\lambda_{2,b-1}=\beta-1$ then $\lambda$ is mapped to itself.
If  $\lambda_{2,b-1}<\beta-1$ then we map $\lambda$ to the partition $\lambda'\in Y_{r,l}^{\beta-1}(m)$ which differs from $\lambda$ only by
$\lambda'_{2,b}=\beta-1$. This can always be done since $\lambda_{1,b}\le\lambda_{1,l} \le l+m-r< \beta$ by assumption. The inverse of the
bijection is defined in the same way.
\end{remark}

\section{Three point correlations} \label{S:threept}
In line with \cite{AAV}, we also present results for three point correlations. Just as in \cite[Section 7.4]{AAV}, we will only prove results for three consecutive particles. Without loss of generality, we calculate the probability of particles of  three different species being at sites 1, 2 and 3. What differentiates these cases is the relative order of species. 
For convenience we will denote the lowest of these species by $i$, the next by $j$, and the highest by $k$.
It is then clear that there are $3!$ different possibilities, the permutations of the letters $i,j,k$. Let $\pi$ be any such permutation. We fix $n$ and denote $E_{\pi}$ to be the probability in the stationary distribution of the TASEP to have the $\pi_{i}$'th particle at site $i$ for $i=1,2,3$. 

The proofs will sum over the number of SSYT with three columns of lengths $a\ge b\ge c$. We will repeatedly use the following formula that is easy to 
deduce from the hook-content formula (see \cite[Corollary  7.21.4]{EC2})
\be \label{hookcontent3}
\ssyt_{a,b,c}(m)=
\begin{cases} 
\frac{(a-b+1)(a-c+2)(b-c+1)}{(a+1)(a+2)(b+1)} & \\
\times \binom{m}{a}\binom{m+1}{b}\binom{m+2}{c}, & 
\text{ if $m\ge a\ge b\ge c\ge 0$} \\
0, & \text { otherwise.}
\end{cases}
\ee 
We will also use the convention that $\ssyt_{0,0,0}(m)=1$ for all $m\ge 0$. 

The proofs in this section will follow the same lines as the proof of Theorem \ref{T:cji}. Here we will project to the four species system 
$\om_{{m}_{r,s,t}}$, where ${m}=(r,s,t,n-r-s-t)$. That is, to a system with $r$ 1s, $s$ 2s, $t$ 3s and $n-r-s-t$ particles of class 4. 
The case with the most uniform answer is the decreasing case, corresponding to the permutation $321$.

\begin{theorem} \label{T:321corr}
For $i<j<k$,
\[
E_{k,j,i} = \frac{6(j-i)(k-i)(k-j)}{n^{3}(n-1)^{2} (n-2)}.
\]
\end{theorem}

\begin{proof} 
Let $T_{r,s,t}=\P(w_1=4,w_2=3,w_3=2)$ be the probability at stationarity in the Markov chain  $\om_{{m}_{r,s,t}}$ of having the first positions occuopied by $4,3,2$. We will compute $T_{r,s,t}$ in two different ways. First by the projection principle we have for all $0\le r,s,t<n$ with $r+s+t\le n$, 
\be \label{Trst}
T_{r,s,t}=\sum_{k=r+s+t+1}^n\sum_{j=r+s+1}^{r+s+t}\sum_{i=r+1}^{r+s} E_{k,j,i}.
\ee
We can evaluate the sum  explicitly by substituting the value of the summand. Thus it suffices to prove that 
\be\label{E:321}
T_{r,s,t}=\!\! \sum_{k,j,i}  \frac{6(j-i)(k-i)(k-j)}{n^{3}(n-1)^{2} (n-2)}=\frac{st(s+t)(n-r)(n-r-s)(n-r-s-\! t)}{n^{3}(n-1)^{2} (n-2)}. 
\ee
Note that we can, by inclusion-exclusion, obtain the $E_{k,j,i}$'s from $T_{r,s,t}$'s. 

Second we can compute $T_{r,s,t}$ by counting the number of multiline queues that gives a word starting with $4,3,2$. The only possibility is
\be \label{E:que432}
\begin{array}{c c c c c c} \vac & \vac & \vac  & . & \dots & .\\ 
\vac & \vac & \occ & . & \dots & .\\ 
\vac & \occ &  \occ & . & \dots & .\\
\hline 4 &3 & 2 &  . & \dots & .\end{array}.
\ee
No particles may be queueing in the beginning of this multi line queue and thus we get the following set of inequalities. Let $z_{x,y}$ is the position of the $x$'th vacant position in row $y$.  
\[
\begin{array}{cccccccc}
4  \le z_{1,4} < & \dots & <z_{1,s+3}< & \cdots  & < z_{1,s+t+2}< & \cdots & < z_{1,n-r} & \le n\\
   & & \text{\begin{sideways} $\le$  \end{sideways}} &    \dots  & \text{\begin{sideways} $\le$  \end{sideways}} &  $\dots$ & 
  \text{\begin{sideways} $\le$  \end{sideways}} &  \\ 
 &4 & \le z_{2,3}<  & \dots  & <z_{2,t+2}< & \cdots & < z_{2,n-r-s} & \le n\\
 &&&&  \text{\begin{sideways} $\le$  \end{sideways}} &    \dots & \text{\begin{sideways} $\le$  \end{sideways}} &  \\
 && & 4 & \le z_{3,2} < & \cdots & < z_{3,n-r-s-t} & \le n.\\
\end{array}
\]
As in the proof of Theorem~\ref{T:cji} the number of possible values of the $z_{x,y}$ is counted by the number of semistandard Young tableaux, here
\[
\ssyt_{n-r-3,n-r-s-2,n-r-s-t-1}(n-3).
\]
The total number of multiline queues is $\binom{n}{r}\binom{n}{r+s}\binom{n}{r+s+t}$, since we may choose the
$r$ occupied positions in the first row, the $r+s$ occupied positions in the second row and the $r+s+t$ occupied positions in the third row in all possible ways. Thus the probability $T_{r,s,t}$ is the number of SSYT divided by this product of binomials. Using \eqref{hookcontent3} we get

\begin{align*}
T_{r,s,t}=&\frac{\binom{n-3}{r}\binom{n-2}{r+s}\binom{n-1}{r+s+t}}{\binom{n}{r}\binom{n}{r+s}\binom{n}{r+s+t}}\frac{st(s+t)}{(n-r-2)(n-r-1)(n-r-s-1)}\\
=&\frac{(n-r)(n-r-s)(n-r-s-t)st(s+t)}{n(n-1)(n-2)n(n-1)n},
\end{align*}
which is equal to \eqref{E:321}. 

To complete the proof, one needs to check that also the special cases $r=0, s=0$ and $t=0$ also satisfy the identity. In principle, these need to be done separately because these will involve projections to multiline queues with less than 4 queues. When $s$ or $t$ are zero, $T_{r,s,t}$ is identically zero as needed. The case of $r=0$ needs to be checked separately and this can be done just as in the special case of Lemma~\ref{L:xylesseqn}.
\end{proof}

The next case we consider is the permutation $213$. This also gives us the answer for the permutation $132$ by particle-hole symmetry Proposition~\ref{P:TASEPbasic}(ii), $E_{j,i,k}=E_{n+1-k,n+1-i,n+1-j}$.

\begin{theorem} \label{T:213corr}
For $i<j<k$,
\[
E_{j,i,k} = \begin{cases}
\ds \frac{2(j-i)}{n^3 (n-1)}, & k>j+1, \\
\\
\ds \frac{2(j-i)}{n(n-1)}
\left( \frac{1}{n^2} + \frac{j(n-j)}{n^2(n-1)} \right) & \\
\\
\ds \quad + \frac{2j(j-1)(n-j)}{n^3(n-1)^2(n-2)} & k=j+1.
\end{cases}
\]
\end{theorem}

\begin{proof}
This time we define $T_{r,s,t}=\P(w_1=3,w_2=2,w_3=4)$. By the projection principle we can first use
\be
T_{r,s,t}=\sum_{k=r+s+t+1}^n\sum_{j=r+s+1}^{r+s+t}\sum_{i=r+1}^{r+s} E_{j,i,k}. 
\ee

Summing over the expression we want to prove we obtain, using a computer algebra package,
\be\label{E:213}
\textstyle T_{r,s,t}=\frac{s(n-r-s-t)(n^2st+n^2t^2+2nrt+nrs+ns^2-ns-nt-nr-nt^2+2rs+2r^2)}{n^3(n-1)^2(n-2)}.
\ee

Second, we need to count the number of multi line queues that makes the TASEP word start with 324
There are four possible different configurations

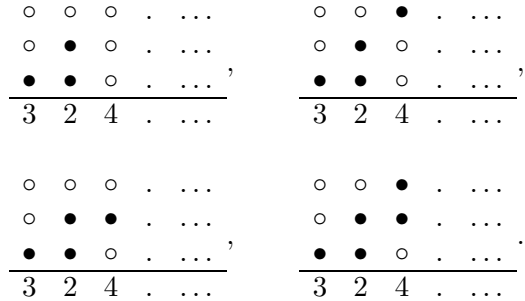
\begin{center}
\begin{tabular}{c c}
$\begin{array}{c c c c l} 
\vac & \vac & \vac  & . & \dots \\ 
\vac & \occ & \vac & . & \dots \\ 
\occ & \occ &  \vac & . & \dots \\
\hline 3 &2 & 4 &  . & \dots \end{array}$, \quad\quad & 
$\begin{array}{c c c c l} 
\vac & \vac & \occ  & . & \dots \\ 
\vac & \occ & \vac & . & \dots \\ 
\occ & \occ &  \vac & . & \dots \\
\hline 3 &2 & 4 &  . & \dots \end{array}$, \\
\\
$\begin{array}{c c c c l} 
\vac & \vac & \vac  & . & \dots \\ 
\vac & \occ & \occ & . & \dots \\ 
\occ & \occ &  \vac & . & \dots \\
\hline 3 &2 & 4 &  . & \dots \end{array}$, \quad\quad &
$\begin{array}{c c c c l} 
\vac & \vac & \occ  & . & \dots \\ 
\vac & \occ & \occ & . & \dots \\ 
\occ & \occ &  \vac & . & \dots \\
\hline 3 &2 & 4 &  . & \dots \end{array}$.
\end{tabular}

\end{center}

These are, by arguments similar to above, counted by 
\be
\sum_{x=0}^1\sum_{y=0}^1 \ssyt_{n-r-2-x,n-r-s-1-y,n-r-s-t-1}(n-3).
\ee
Thus, using \eqref{hookcontent3}, we get
\begin{align*}
&\binom{n}{r}\binom{n}{r+s}\binom{n}{r+s+t} T_{r,s,t}\\
=&\sum_{x=0}^1\sum_{y=0}^1 \binom{n-3}{r+x-1}\binom{n-2}{r+s+y-1}\binom{n-1}{r+s+t} \\
& \times \frac{(s+y-x)(t-y+1)(s+t-x+1)}{(n-r-x-1)(n-r-x)(n-r-s-y)}.\\
\end{align*}
Performing the sums and simplifying, $T_{r,s,t}$ becomes equal to the expression in \eqref{E:213}, completing the proof.
\end{proof}

The last case for which we can prove the formula for correlations in $312$, and by particle-hole symmetry Proposition~\ref{P:TASEPbasic}(ii) $231$, that is $E_{j,k,i}=E_{n+1-i,n+1-k,n+1-j}$.

\begin{theorem} \label{T:231corr}
For $i<j<k$,
\[
E_{j,k,i} = \begin{cases}
\ds \frac{3(j-i)(2n-j-i-1)}{n^3 (n-1)(n-2)}
- \frac{4(j-i)(n-k)}{n^3 (n-1)^2}, & k>j+1, \\
\\
\ds \frac{(j-i)(n-1-j)}{n^2 (n-1)^2}
\left( \frac{1}{n-2} +\frac{3(n-i-1)}{n} \right.  & \\
\\
\ds  \left. -\frac{(n-1-j)(3n-3i+j-1)}{n (n-2)} \right) +\frac{6(j-i)(n-i)}{n^3 (n-1)(n-2)} , & k=j+1.
\end{cases}
\]
\end{theorem}

\begin{proof}
This time we define $T_{r,s,t}=\P(w_1=3,w_2=4,w_3=2)$.
Summing $E_{j,k,i}$'s just as we did in \eqref{Trst}, we get again using a standard computer algebra package,
\be \label{E:231}
\begin{split}
T_{r,s,t}=&
s(n-r-s-t)\Big(2n^2st+2n^2t^2+nrs-ns-nt-nt^2s-ns^2t\\
&-2nstr+2nrt+ns^2-2nt^2r-nt^2-nr+r^2\\
&+rs-2trs-r^2s-rs^2-2tr^2\Big)/\Big(n^3(n-1)^2(n-2)\Big).
\end{split}
\ee

Now, we need to count the number of multiline queues that make the TASEP word start with 342. There are four possible configurations
\begin{center}
\begin{tabular}{c c}
$\begin{array}{c c c c l} 
\vac & \vac & \vac  & . & \dots \\ 
\vac & \occ & \vac & . & \dots \\ 
\occ & \vac &  \occ & . & \dots \\
\hline 3 &4 & 2 &  . & \dots \end{array}$,  \quad\quad &
$\begin{array}{c c c c l} 
\vac & \vac & \vac  & . & \dots \\ 
\vac & \vac & \occ & . & \dots \\ 
\occ & \vac &  \occ & . & \dots \\
\hline 3 &4 & 2 &  . & \dots \end{array}$,\\
\\
$\begin{array}{c c c c l} 
\vac & \vac & \occ  & . & \dots \\ 
\vac & \occ & \vac & . & \dots \\ 
\occ & \vac &  \occ & . & \dots \\
\hline 3 &4 & 2 &  . & \dots \end{array}$,  \quad\quad &
$\begin{array}{c c c c l} 
\vac & \vac & \vac  & . & \dots \\ 
\vac & \occ & \occ & . & \dots \\ 
\occ & \vac &  \occ & . & \dots \\
\hline 3 &4 & 2 &  . & \dots \end{array}$,
\end{tabular}
\end{center}

\smallskip
\noindent
where, as before, no particle can be in queue in any of the rows in the first column.
These are counted by 
\begin{align*}
& 2\, \ssyt_{n-r-3,n-r-s-2,n-r-s-t-1}(n-3) \\
&+\ssyt_{n-r-2,n-r-s-2,n-r-s-t-1}(n-3)\\
&+\ssyt_{n-r-3,n-r-s-1,n-r-s-t-1}(n-3).
\end{align*}
Using \eqref{hookcontent3} and dividing by $\binom{n}{r}\binom{n}{r+s}\binom{n}{r+s+t}$  we get
\begin{align*}
T_{r,s,t}=&
\frac{(n-r)(n-r-s)(n-r-s-t)}{n^3(n-1)^2(n-2)} \\
&\left (2st(s+t)+\frac{(s+1)t(s+t+1)r}{n-r}+\frac{(s-1)(t+1)(s+t)}{n-r-s} \right ),\\
\end{align*}
and a simple calculation shows that this is equal to \eqref{E:231}.
\end{proof}

For the increasing case we have not been able to compute the correlations using this method. This is in one sense the most interesting
formula because the probability is completely independent of $i,j,k$ when they are far apart. This generalizes the behaviour of $E_{i,j}$ in Theorem~\ref{T:cji}.

To prove this conjecture with the same approach as for the other patterns, we would need a formula for the number of SSYT and for the 
number of near-SSYT, that is tableaux that are not SSYT but becomes SSYT if the last column is moved one step up (and the top value deleted).

\begin{conjecture} \label{conj:123corr}
For $i<j<k$,
\[
E_{i,j,k} = \begin{cases}	
\ds \frac 1{n^{3}} & i<j-1<k-2, \\
\\
\ds \frac{n-1+i(n-i)}{n^{3}(n-1)} & i=j-1<k-2, \\
\\
\ds \frac{n-1+j(n-j)}{n^{3}(n-1)} & i<j-1=k-2, \\
\\
\ds \frac{(n-1+i(n-i))(n-1+(i+1)(n-i-1))}{n^{3}(n-1)^{2}}\! & \\
\\
\ds \quad\quad+\frac{2i(i+1)(n-i)(n-i-1)}{n^{3}(n-1)^{2}(n-2)}& i=j-1=k-2.
\end{cases}
\]
\end{conjecture}

Using Theorems~\ref{T:321corr}, \ref{T:213corr} and \ref{T:231corr} and
assuming Conjecture~\ref{conj:123corr}, one may deduce as a corollary similar to Corollary~\ref{C:density2}, the joint distributions of three 
consecutive particles in the TASEP speed 
process \cite[Theorem 7.7]{AAV}. This should come as no surprise; the proof in \cite{AAV} also amounts to projecting the TASEP to four-particle 
systems and then studying the multiline queues. 

\begin{corollary}[Amir-Angel-Valk\'o \cite{AAV}]
\label{C:density3}
Assuming Conjecture~\ref{conj:123corr}, in the limit as $i, j ,k,n \to \infty$ with $i \leq j \leq k$ and
\[
\frac {i}n \to \frac{x+1}2,
\frac {j}n \to \frac{y+1}2 \text{ and }
\frac {k}n \to \frac{z+1}2,
\]
so that $x,y,z \in [-1,1]$, the joint densities become exactly those given in 
Table 2 of \cite{AAV}.
\end{corollary}

\begin{proof}
The idea of the proof is essentially identical to that of Corollary~\ref{C:density2} and we will not repeat all the details. In the first part of the table, one obtains the density by multiplying the probability mass functions by $(n/2)^3$ and taking the limit. In the second and third parts, one
 does the same, except that the prefactor becomes $(n/2)^2$ and $(n/2)$ respectively.
\end{proof}

We do however, as in Corollary \ref{C:addprob12}, get  a finite strengthening for the probability in the 9 cases corresponding to Theorems~\ref{T:321corr}, \ref{T:213corr} and \ref{T:231corr} and conjecture one for the four cases of Conjecture~\ref{conj:123corr}. The proof is just a summation over all possibilities for each case.

\begin{corollary}
Assuming Conjecture~\ref{conj:123corr}, the probability of all cases for three adjacent positions in the TASEP on permutations is given by Table \ref{Table:prob}. It is assumed that 
$i<j<k$ and the more general cases do not cover the more specific cases mentioned.
\end{corollary}
\begin{table}[htbp!]
\[
\begin{array}{|c | c| c |}
\hline
(w_1,w_2,w_3) & \text{Probability} & \text{Follows from}\\
\hline
\noalign{\smallskip}
 (k,j,i) & 
 \frac{(n+1)(n+2)}{30n(n-1)} = \frac{1}{30}+\frac{2 n+1}{15n (n-1)}
 & \text{Theorem \ref{T:321corr}} \\
\hline
\noalign{\smallskip}
(i,k,j) & 
\multirow{2}{*}{$\frac{(n-2)(n-3)}{12n^2} = \frac{1}{12}-\frac{5 n-6}{12 n^2}$}
& 
\multirow{2}{*}{\text{Theorem \ref{T:213corr}}} \\
\;\&\; (j,i,k)  & & \\
\hline
\noalign{\smallskip}
(i,k,i+1) & 
\multirow{2}{*}{$\frac{n^3+8n^2-23n+10}{20n^2(n-1)} 
= \frac{1}{20} + \frac{9 n^2-23 n+10}{20 n^2 (n-1)}$} & 
\multirow{2}{*}{\text{Theorem \ref{T:213corr}}} \\
\;\&\; (j,i,j+1)  & & \\
\hline
\noalign{\smallskip}
(k,i,j) & 
\multirow{2}{*}{$\frac{(n+1)(n-3)(7n-10)}{60n^2(n-1)}
=\frac{7}{60}-\frac{17 n^2+n-30}{60 n^2 (n-1)}$} & 
\multirow{2}{*}{\text{Theorem \ref{T:231corr}}} \\
\;\&\;(j,k,i) & & \\
\hline
\noalign{\smallskip}
(k,i,i+1)  & 
\multirow{2}{*}{$\frac{(n+1)(n^2+7n-10)}{20n^2(n-1)} 
= \frac{1}{20} + \frac{9 n^2-3 n-10}{20 n^2  (n-1)}$} &
\multirow{2}{*}{\text{Theorem \ref{T:231corr}}} \\ 
\;\&\; (j,j+1,i) & & \\
\hline
\noalign{\smallskip}
(i,j,k) & \frac{(n-2)(n-3)(n-4)}{6n^3} = \frac{1}{6} + 
\frac{9 n^2-26 n+24}{6 n^3}& \text{Conjecture \ref{conj:123corr}} \\
\hline
\noalign{\smallskip}
(i,i+1,k)  & 
\multirow{2}{*}{$\frac{(n-2)(n-3)(n+8)}{12n^3}
=\frac{1}{12}+\frac{3 n^2-34 n+48}{12 n^3}$} &  
\multirow{2}{*}{\text{Conjecture \ref{conj:123corr}}} \\
\;\&\; (i,j,j+1) & & \\
\hline
\noalign{\smallskip}
(i,i+1,i+2) & \frac{n^4+13n^3+32n^2-160n+120}{30n^3(n-1)}
=\frac{1}{30}+\frac{7 n^3+16 n^2-80 n+60}{15n^3 (n-1)}  & 
\text{Conjecture \ref{conj:123corr}}  \\
\hline 
\end{array}
\]
\caption{Correlations $E_{w_1,w_2,w_3}$ in the limit.}
\label{Table:prob}
\end{table}

\section{Discussions and open problems}
\label{S:open}

As expected the correlation of several particles seems intractable in general. We can however state one fact and several conjectures.
We denote, just as in \refS{S:threept}, the general nearest neighbour correlation $E_{i_1,\dots,i_r}$ to be the joint probability of seeing 
$i_a$ in position $a$, for $a=1,\dots,r$. For the decreasing case we conjectured the following Vandermonde formula. A proof of this using 
determinantal techniques will appear in \cite{AasL}, but it is natural to ask for a simple proof of this result.

\begin{theorem} \label{T:decreasing}
For $r \leq n$ and $i_1 > \cdots > i_r$,
\[
E_{i_1,i_2\dots,i_r} = r! \; \frac{\ds \prod_{1 \leq a<b \leq r} (i_a-i_b)}
{\ds \;\;\prod_{i=0}^{r-1} (n-i)^{r-i}}.
\]
\end{theorem}

The writing of this paper has been delayed partly because we really wanted to prove the $1/n^3$ formula in 
Conjecture \ref{conj:123corr}. It is 
interesting  that the simplest three-point correlation formula has proved impervious to our proof technique. 
The $1/n^3$ formula also strongly suggests that there is some more 
conceptual independence to be discovered and proved for the cyclic multispecies TASEP.  
We offer a sequence of conjectures generalizing this observation.

\begin{conjecture} \label{conj:increasing} 
For $i_1<i_2-1<i_3-2<\dots <i_r-(r-1)$,
\[
E_{i_1,\dots,i_r} =\frac 1{n^r}.
\]
\end{conjecture}

Note that the first formula in \refT{T:213corr} says $E_{j,i,k}=E_{j,i}\cdot \frac{1}{n}$ if $j+1<k$. This naturally suggest the following generalization of 
Conjecture \ref{conj:increasing}.

\begin{conjecture} \label{conj:one} 
For any $i_1, i_2,\dots, i_r$ and $k > 1+ \max_a i_a$, then
\[
E_{i_1,\dots,i_r,k} =\frac 1{n}\cdot E_{i_1,\dots,i_r}.
\]
\end{conjecture}

Note that this is a finite cyclic version of \cite[Lemma 7.3]{AAV}.
Inspired by our study of correlation of particles further apart in \refS{S:Two far}, in particular Corollary \ref{C:uniform} 
there is reason to believe the following natural generalization of Conjecture \ref{conj:one}. 

\begin{conjecture} \label{conj:oneapart} 
For any $i_1, i_2,\dots, i_r$ and $b,k$ such that $k > b-r+ \max_a i_a$, then
\[
\P(w_{a}=i_a \text { for } 1\le a\le r, \text { and } w_b=k) =\frac 1{n}\cdot E_{i_1,\dots,i_r}
\]
\end{conjecture}

We end with the most general independence for nearest neighbour correlations in two blocks.

\begin{conjecture} \label{conj:many} 
For any $i_1, \dots, i_r, j_{r+1},\dots, j_{r+s}$, such that for all $r+1\le b\le r+s$ and  $\min_b j_b>1 +\max_a i_a$,
we have the independence
\[
\P(w_{a}=i_a, 1\le a\le r, \text { and } w_{b}=j_b, r+1\le b\le r+s) =
E_{i_1,\dots,i_r}E_{j_{r+1},\dots,j_{r+s}}. 
\]
\end{conjecture}

The reader should compare this to \cite[Corollary 2.4]{AS}, where such independence is proved if $r+s=n$ and 
$\min_b j_b>\max_a i_a$. 
All the above conjectures fit with data for values of $n$ upto 8.

The techniques used to prove the two-point case in Theorem~\ref{T:cji} fail for Conjecture~\ref{conj:123corr} and it is clear that new 
ideas are needed for the conjectures above.
Maybe the key to all these conjectures is to find a conceptual proof of $E_{i,j}=1/n^2$ if $j>i+1$ in \refT{T:cji}. 

\begin{problem}
Find a more conceptual proof of the independence $E_{i,j}=1/n^2$ if $j>i+1$. 
\end{problem}

On the combinatorial side it would be interesting to find a simpler proof of Theorem \ref{T:Y2}.

\begin{problem}
Find a proof of \refT{T:Y2} that explains the occurrence of the Narayana numbers.
\end{problem}

\end{document}